\documentclass[a4paper, oneside,11pt]{article}

\DeclareMathSizes{11}{19}{13}{9}

\usepackage{amsmath,amssymb,amsthm}

\usepackage{geometry}
\usepackage{times}

\usepackage{bm}
\usepackage{bbm}
\usepackage{natbib}
\usepackage{enumitem}
\usepackage{nicefrac}
\usepackage{graphicx}
\usepackage{caption}
\usepackage{subcaption}

\usepackage[plain,noend]{algorithm2e}

\theoremstyle{plain}
\newtheorem{theorem}{Theorem}%
\newtheorem{corollary}[theorem]{Corollary}%
\newtheorem{lemma}[theorem]{Lemma}%
\theoremstyle{definition}
\newtheorem{remark}[theorem]{Remark}%
\newtheorem*{remark*}{Remark}%
\renewcommand\theta{\vartheta}

\def\dif{\mathrm{d}}
\def\R{\mathbb{R}}
\def\N{\mathbb{N}}

\def\cov{\mathrm{cov}}
\def\prob{\mathrm{P}}
\def\cF{\mathcal{F}}
\newcommand{\E}[1]{\mathbb{E}\left[ #1 \right]}
\newcommand{\En}[1]{\mathbb{E}_n\left[ #1 \right]}
\newcommand{\Ens}[1]{\mathbb{E}_n^*\left[ #1 \right]}

\newcommand{\ind}[1]{\mathbbm{1}\left( #1 \right)}
\def\Inv{\scriptscriptstyle \mathrm{Inv}}

\begin{document}

	\markboth{T. Zwingmann and H. Holzmann}{Weak convergence of expectile processes}
	
	\title{Weak convergence of quantile and expectile processes under general assumptions}
	
	\author{Tobias Zwingmann\footnote{Email: zwingmann@mathematik.uni-marburg.de} \ and         Hajo Holzmann\footnote{Corresponding author, Email: holzmann@mathematik.uni-marburg.de}  \\
		\small{Fachbereich Mathematik und Informatik}  \\
		\small{Philipps-Universit\"at Marburg, Germany} \\
	}
\maketitle

\begin{abstract}
We show weak convergence of quantile and expectile processes to Gaussian limit processes in the space of bounded functions endowed with an appropriate semimetric which is based on the concepts of epi- and hypo convergence as introduced in \citet{buecher2014}. We impose assumptions for which it is known that weak convergence with respect to the supremum norm or the Skorodhod metric generally fails to hold. For expectiles, we only require a distribution with finite second moment but no further smoothness properties of distribution function, for quantiles, the distribution is assumed to be absolutely continuous with a version of its Lebesgue density which is strictly positive and has left- and right-sided limits. We also show consistency of the bootstrap for this mode of convergence.   
\end{abstract}

\section{Introduction}

Quantiles and expectiles are fundamental parameters of a distribution which are of major interest in statistics, econometrics and finance \citep{Koenker, Klar1, Newey, Ziegel1} . 

The asymptotic properties of sample quantiles and expectiles have been addressed in detail under suitable conditions. 
For quantiles, differentiability of the distribution function at the quantile with positive derivative implies asymptotic normality of the empirical quantile, and under a continuity assumption on the density one obtains weak convergence of the quantile process to a Gaussian limit process in the space of bounded functions with the supremum distance from the functional delta-method \citep{vdv2000}. However, without a positive derivative of the distribution function at the quantile, the weak limit will be non-normal \citep{Knight2002LimDistr}, and thus process convergence to a Gaussian limit with respect to the supremum distance cannot hold true.  

Similarly, for a distribution with finite second moment, the empirical expectile is asymptotically normally distributed if the distribution function is continuous at the expectile, but non-normally distributed otherwise \citep{holzmann2016}. For continuous distribution functions, process convergence of the empirical expectile process in the space of continuous functions also holds true, but for discontinuous distribution functions this can no longer be valid. 

In this note we discuss convergence of quantile and expectile processes from independent and identically distributed observations  under more general conditions. Indeed, we show that the expectile process converges to a Gaussian limit in the semimetric space of bounded functions endowed with the hypi-semimetric as recently introduced in \citet{buecher2014} under the assumption of a finite second moment only. Since the Gaussian limit process is discontinuous in general while the empirical expectile process is continuous, this convergence can hold neither with respect to the supremum distance nor with a variant of the Skorodhod metric. Similarly, we show weak convergence of the quantile process under the hypi-semimetric if the distribution function is absolutely continuous with a version of its Lebesgue density that is strictly positive and has left and right limits at any point. These results still imply weak convergence of important statistics such as Kolmogorov-Smirnov and Cramer-von Mises type statistics.  
We also show consistency of the $n$ out of $n$ bootstrap in both situations. 

A sequence of bounded functions $(f_n)_{n \in \N}$ on a compact metric space hypi-converges to the bounded function $f$ if it epi-converges to the lower\--se\-mi\-con\-tin\-u\-ous hull of $f$, and hypo-converges to its upper\--se\-mi\-con\-tin\-u\-ous hull, see the supplementary material for the definitions. \citet{buecher2014} show that this mode of convergence can be expressed in terms of a semimetric $\dif_{hypi}$ on the space of bounded functions, the hypi-semimetric.   

We shall use the notation $c^{-1}=1/c$ for $c\in\R$, $c\ne 0$, and denote the (pseudo-) inverse of a function $h$ with $h^{\Inv}$. We write $\stackrel{\mathcal{L}}{\to}$ for ordinary weak convergence of real-valued random variables. Weak convergence in semimetric spaces will be understood in the sense of Hofmann-J\o rgensen, see \citet{vdvW2013} and \citet{buecher2014}.

\section{Convergence of the expectile process}

For a random variable $Y$ with distribution function $F$ and finite mean $\E{|Y|}<\infty$, the 
$\tau$-expectile $\mu_{\tau} = \mu_{\tau} (F)$, $\tau \in (0,1)$, can be defined as the unique solution of  
%
$	\E{I_\tau(x,Y)} = 0,$ $x \in \R, $
%
where 
\begin{equation}\label{eq:identfct}
	I_\tau(x,y) = \tau (y-x) \ind{y \geq x} - (1-\tau)\, (x-y) \ind{y < x}, 
\end{equation}
and $\ind{\cdot}$ is the indicator function. Given a sequence of independent and identically distributed copies $Y_1, Y_2, \ldots $ of $Y$ and a natural number $n \in \N$,  we let 
\[
 \hat \mu_{\tau,n} = \mu_\tau\big(F_n\big), \qquad F_n(x) = \frac1n \sum_{k=1}^n\, \ind{Y_k \leq x},
\]
be the empirical $\tau$-expectile and the empirical distribution function, respectively. 

\begin{theorem}\label{thm:main_thm}
	Suppose that $\E{Y^2} < \infty$. Given $0<\tau_l<\tau_u<1$, the standardized expectile process $\tau \mapsto \sqrt{n}\big(\hat \mu_{\tau, n}-\mu_{\tau}\big)$, $\tau \in [\tau_l, \tau_u]$, converges weakly in $\big(\ell^{\infty}[\tau_l,\tau_u], \dif_{hypi}\big)$ to the limit process $\big(\dot{\psi}^{\Inv}(Z)(\tau)\big)_{\tau\in [\tau_l, \tau_u]}$. Here, 
$ 
	\dot{\psi}^{\Inv} (\varphi) (\tau)  = \big(\tau + (1-2\,\tau) F(\mu_\tau)\big)^{-1} \varphi(\tau)
$, 
$ \varphi \in \ell^\infty[\tau_l, \tau_u]$, 
and	
	$(Z_\tau)_{\tau \in [\tau_l,\tau_u]}$ is a centered tight Gaussian process with continuous sample paths and covariance function $\cov\big(Z_{\tau},Z_{\tau'}\big) = \E{I_{\tau}(\mu_{\tau},Y)\,I_{\tau'}(\mu_{\tau'},Y)}$ for $\tau,\tau'\in[\tau_l,\tau_u]$.
\end{theorem}

From Propositions~2.3 and 2.4 in \citet{buecher2014}, hypi-con\-ver\-gence of the expectile process implies ordinary weak convergence of important statistics such as Kol\-mo\-go\-rov-Smirnov or Cram\'{e}r-von Mises type statistics. We let 
\[ \|\varphi\|_{[\tau_l, \tau_u]}= \sup_{\tau \in [\tau_l, \tau_u]} |\varphi(\tau)|, \quad \varphi \in \ell^\infty[\tau_l, \tau_u],\]
denote the supremum norm on $\ell^\infty[\tau_l, \tau_u]$. 
\begin{corollary}\label{cor:derivedstatistics}
If $\E{Y^2} < \infty$, then we have as $n \to \infty$ that
\begin{equation*}
		\sqrt{n}\big\Vert\hat{\mu}_{\cdot,n}-\mu_{\cdot} \big\Vert_{[\tau_l, \tau_u]} \stackrel{\mathcal{L}}{\to} \big\Vert \dot{\psi}^{\Inv}\big(Z\big) \big\Vert_{[\tau_l, \tau_u]}.
	\end{equation*}
Further, for $p \geq 1$ and a bounded, non-negative weight function $w$ on $[\tau_l, \tau_u]$, 
\begin{equation*}
		n^{p/2}\, \int_{\tau_l}^{\tau_u} \big|\big(\hat{\mu}_{\tau,n}-\mu_{\tau}\big)\big|^p\, w(\tau)\, \mathrm{d} \tau 
		\stackrel{\mathcal{L}}{\to} 
		\int_{\tau_l}^{\tau_u} \big|\big(\dot{\psi}^{\Inv}(Z)\big)(\tau) \big|^p\, w(\tau)\, \dif \tau .
\end{equation*}

\end{corollary}

\begin{remark}\label{rem:pointeval}
Evaluation at a given point $x$ is only a continuous operation under the hypi-semi\-me\-tric if the limit function is continuous at $x$, see Proposition 2.2 in \citet{buecher2014}. In particular, this does not apply to the expectile process if the distribution function $F$ is discontinuous at $\mu_\tau$. Indeed, Theorem 7 in \citet{holzmann2016} shows that the weak limit of the empirical expectile is not normal in this case. 
\end{remark}
Next we turn to the validity of the bootstrap. Given $n \in \N$ let $Y_1^*, \ldots ,Y_n^*$ denote a sample drawn from $Y_1, \ldots, Y_n$ with replacement, that is, having distribution function $F_n$. Let $F_n^*$ denote the empirical distribution function of $Y_1^*, \ldots ,Y_n^*$, and let $\mu_{\tau,n}^* = \mu_\tau\big(F_n^*\big)$ denote the bootstrap expectile at level $\tau \in (0,1)$. 

\begin{theorem}\label{thm:main_thm_boot}
	Suppose that $\E{Y^2} < \infty$. Then, almost surely, conditionally on $Y_1, Y_2, \ldots $ the standardized bootstrap expectile process $\tau \mapsto \sqrt{n}\big( \mu_{\tau, n}^*-\hat \mu_{\tau,n}\big)$, $\tau \in [\tau_l, \tau_u]$, converges weakly in $\big(L^{\infty}[\tau_l,\tau_u], \dif_{hypi}\big)$ to $\big(\dot{\psi}^{\Inv}(Z)(\tau)\big)_{\tau\in[\tau_l,\tau_u]}$, where the map $\dot{\psi}^{\Inv} $ and the process $(Z_\tau)_{\tau\in[\tau_l,\tau_u]}$ are as in Theorem \ref{thm:main_thm}. 
\end{theorem}

\begin{remark*}
The simple $n$ out of $n$ bootstrap does not apply for the empirical expectile at level $\tau$ if $F$ is discontinuous at $\mu_\tau$, see \citet{knightboot} for a closely related result for the quantile. Thus, Theorem \ref{thm:main_thm_boot} is somewhat surprising, but its conclusion is reasonable in view of Remark \ref{rem:pointeval}.  
\end{remark*}

\section{Convergence of the quantile process}

Let $q_\alpha = q_\alpha(F) = \inf\{ x \in \R:\, F(x) \geq \alpha\}$, $\alpha \in (0,1)$, denote the $\alpha$-quantile of the distribution function $F$, and let $\hat q_\alpha = q_\alpha(F_n)$ be the empirical $\alpha$-quantile of the sample $Y_1, \ldots, Y_n$. 

\begin{theorem}\label{thm:conv_quantile}
	Suppose that the distribution function $F$ is absolutely continuous and strictly increasing on $[q_{\alpha_l}, q_{\alpha_u}]$, $0 < \alpha_l < \alpha_u < 1$. Assume that there is a version $f$ of its density which is bounded and bounded away from zero on $[q_{\alpha_l}, q_{\alpha_u}]$, and which admits right- and left-sided limits in every point of $[q_{\alpha_l}, q_{\alpha_u}]$.

	Then the standardized quantile process $\alpha \mapsto \sqrt{n}\big(\hat{q}_{\alpha,n}-q_{\alpha}\big)$, $\alpha \in [\alpha_l ,\alpha_u]$, converges weakly in $(L^{\infty}[\alpha_l,\alpha_u], \dif_{hypi})$ to the process $\big(\nicefrac{V_{\alpha}}{f(q_{\alpha})}\big)_{\alpha\in[\alpha_l,\alpha_u]}$, where 
	$(V_\alpha)_{\alpha\in[\alpha_l,\alpha_u]}$ is a Brownian bridge on $[\alpha_l,\alpha_u]$.
	
Furthermore, as $n \to \infty$ we have that
\begin{equation*}
		\sqrt{n}\big\Vert\hat{q}_{\cdot,n}-q_{\cdot} \big\Vert_{[\alpha_l, \alpha_u]} \stackrel{\mathcal{L}}{\to} \big\Vert \dot{\Psi}^{\Inv}\big(V\big) \big\Vert_{[\alpha_l, \alpha_u]},
	\end{equation*}
as well as 
\begin{equation*}
		n^{p/2}\, \int_{\alpha_l}^{\alpha_u} \big|\hat{q}_{\alpha,n}-q_{\alpha}\big|^p\, w(\alpha)\, \mathrm{d}\, \alpha
		\stackrel{\mathcal{L}}{\to} 
		\int_{\alpha_l}^{\alpha_u} \big|\big(\dot{\Psi}^{\Inv}(V)\big)(\alpha) \big|^p\, w(\alpha)\, \dif \alpha .
	\end{equation*}
for $p \geq 1$ and a bounded, non-negative weight function $w$ on $[\alpha_l, \alpha_u]$.

\end{theorem}

Let $q_{\alpha,n}^* = q_{\alpha}\big(F_n^*\big)$ denote the bootstrap quantile at level $\alpha \in (0,1)$. 

\begin{theorem}\label{thm:main_boot_quantile}
	Let the assumptions of Theorem~\ref{thm:conv_quantile} be true. Then, almost surely, conditionally on $Y_1, Y_2, \ldots $ the standardized bootstrap quantile process $\alpha \mapsto \sqrt{n}\big( q_{\alpha, n}^*-\hat{q}_{\alpha,n}\big)$, $\alpha \in [\alpha_l, \alpha_u]$, converges weakly in $\big(L^{\infty}[\alpha_l,\alpha_u], \dif_{hypi}\big)$ to $\big(\nicefrac{V_{\alpha}}{f(q_{\alpha})}\big)_{\scriptscriptstyle \alpha\in[\alpha_l,\alpha_u]}$. 
\end{theorem}

\section{Appendix: Outlines of the proofs}


In this section we present an outline of the proofs, technical details can be found in the supplementary material. 

\begin{proof}[Proof of Theorem \ref{thm:main_thm} (Outline).]
We give an outline of the proof of Theorem \ref{thm:main_thm}. For a distribution function $S$ with finite first moment let
$	[\psi(\varphi,S)](\tau) = -I_\tau(\varphi(\tau),S)$, $\tau \in [\tau_l, \tau_u]$ and $ \varphi \in \ell^\infty[\tau_l, \tau_u]$, 
where $I_{\tau}(x,S)=\int I_{\tau}(x,y)\,\dif S(y)$, and set $\psi_0 = \psi(\,,F)$ and $\psi_n = \psi(\,,\hat F_n)$. In the following we simply write $\|\varphi\|$ instead of $\|\varphi\|_{[\tau_l, \tau_u]}$.

\medskip

\emph{Step 1. Weak convergence of $\sqrt{n}\big(\psi_0(\hat{\mu}_{\cdot,n})-\psi_0(\mu_{\cdot})\big)$ to $Z$ in $\big(\ell^\infty[\tau_l, \tau_u], \|\cdot\| \big)$ }

\medskip

This step uses standard results from empirical process theory based on bracketing properties of Lipschitz-continuous functions. The main issue in the proof of the lemma below is the Lipschitz-continuity of $\tau \mapsto \mu_\tau$, $\tau \in [\tau_l, \tau_u]$, for a general distribution function $F$. 
\begin{lemma}\label{lem:expecexapnsion}
In $(\ell^\infty[\tau_l, \tau_u],\Vert\cdot\Vert)$  we have the weak convergence  
	\begin{equation}\label{convergenceprocess}
	\sqrt{n} \big(\psi_n(\mu_\cdot) - \psi_0(\mu_\cdot)  \big)(\tau) \to Z_\tau,\quad \tau \in [\tau_l, \tau_u].
	\end{equation}
Further, given $\delta_n \searrow 0$ we have as $n \to \infty$ that
\begin{align}\label{eq:maximalprocess}
	\begin{split}
	\sup_{\| \varphi\|_{[\tau_l, \tau_u]} \leq \delta_n}\, \sup_{\tau \in [\tau_l, \tau_u]}\, \sqrt{n}\, \big|&\psi_n(\mu_\cdot + \varphi)(\tau) - \psi_0(\mu_\cdot + \varphi)(\tau) \\
	&- \big[ \psi_n(\mu_\cdot)(\tau) - \psi_0(\mu_\cdot)(\tau) \big] \big| = o_\mathrm{P}(1).
	\end{split}
	\end{align}
\end{lemma}
Since $\psi_0(\mu_\cdot) = \psi_n(\hat \mu_{\cdot,n}) = 0$, (\ref{eq:maximalprocess}) and the uniform consistency of $\hat \mu_{\cdot,n}$ from Theorem 1 in \citet{holzmann2016} give 
	\begin{align*}
	%
	\sqrt{n} \big(\psi_0(\hat \mu_{\cdot,n}) - \psi_0(\mu_\cdot)  \big) & = \sqrt{n} \big(\psi_0(\hat \mu_{\cdot,n}) - \psi_n(\hat \mu_{\cdot,n})  \big)\\
		& = - \, \sqrt{n} \big(\psi_n(\mu_{\cdot}) - \psi_0(\mu_{\cdot})  \big) + o_\mathrm{P}(1),
	%
%
	\end{align*}
	where the remainder term is in $(\ell^\infty[\tau_l, \tau_u],\Vert\cdot\Vert)$, 
and (\ref{convergenceprocess}) and the fact that $Z$ and $-Z$ have the same law conclude the proof of 
\begin{equation}\label{eq:firstconv}
\sqrt{n}\big(\psi_0(\hat{\mu}_{\cdot,n})-\psi_0(\mu_{\cdot})\big) \to Z \quad \text{ weakly in } \big(\ell^\infty[\tau_l, \tau_u], \|\cdot\| \big).
\end{equation}

\medskip

{\sl Step 2. Invertibility of $\psi_0$ and semi-Hadamard differentiability of $\psi_0^{\Inv}$ with respect to $\dif_{hypi}$.} 

\medskip

The first part of Step 2. is observing the following lemma. 

\begin{lemma}\label{lem:invertable}
	The map $\psi_0$ is invertible. Further $\psi_0^{\Inv}(\varphi) \in \ell^{\infty}[\tau_l, \tau_u]$ for any $\varphi\in\ell^{\infty}[\tau_l, \tau_u]$.   	
\end{lemma}

%
%

The next result then is the key technical ingredient in the proof of Theorem \ref{thm:main_thm}. 

\begin{lemma}\label{thm:semi_hadamard_diff}
	The map $\psi_0^{\Inv}$ is semi-Hadamard differentiable with respect to the hypi-semimetric in $0\in \mathcal{C}[\tau_l,\tau_u]$ tangentially to $\mathcal{C}[\tau_l,\tau_u]$ with semi-Hadamard derivative $
	\dot{\psi}^{\Inv} (\varphi) (\tau) = \big(\tau + (1-2\,\tau) F(\mu_\tau)\big)^{-1} \varphi(\tau)$, $\varphi \in \ell^{\infty}[\tau_l, \tau_u]$, that is, we have  
	\begin{equation*}
		t_n^{-1}\big(\psi_0^{\Inv}(t_n\,\varphi_n)-\psi_0^{\Inv}(0)\big)\to \dot{\psi}^{\Inv} (\varphi) (\tau)
	\end{equation*}
	for any sequence $t_n \searrow 0$, $t_n>0$ and $\varphi_n\in\ell^{\infty}[\tau_l,\tau_u]$ with $\varphi_n\to \varphi\in\mathcal{C}[\tau_l,\tau_u]$ with respect to $\dif_{hypi}$.
	
\end{lemma}

The proof of the lemma is based on an explicit representation of increments of $\psi_0^{\Inv}$, and novel technical properties of convergence under the hypi-semimetric for products and quotients.

\medskip

\emph{Step 3. Conclusion with the generalized functional delta-method.}

\medskip

From (\ref{eq:firstconv}), Lemma \ref{thm:semi_hadamard_diff} and the generalized functional delta-method, Theorem~B.7 in \citet{buecher2014} we obtain 
\begin{align*}
		\sqrt{n}\big(\hat{\mu}_{\cdot, n}-\mu_{\cdot}\big) = \sqrt{n}\Big(\psi_0^{\Inv}\big(\psi_0(\hat{\mu}_{\cdot,n})\big)-\psi_0^{\Inv}\big(0\big)\Big) \to \dot{\psi}^{\Inv}(Z)	\end{align*}
in $(L^{\infty}[\tau_l,\tau_u], \dif_{hypi})$. 
\end{proof}

\begin{proof}[Proof of Theorem \ref{thm:main_thm_boot} (Outline).]
The steps in the proof are similar to those of Theorem \ref{thm:main_thm}. In the analogous result to Lemma \ref{lem:expecexapnsion} and (\ref{eq:firstconv}), we require the uniform consistency of $\mu_{\cdot,n}^*$ as in \citet{holzmann2016}, Theorem 1. The weak convergence statements require the changing classes central limit theorem, \citet{vdv2000}, Theorem 19.28. 
In the second step we argue directly with the extended continuous mapping theorem, Theorem B.3 in \citet{buecher2014}.
\end{proof}

\begin{proof}[Proof of Theorem \ref{thm:conv_quantile} (Outline).]
Let $G$ be the distribution function of $\mathcal{U}(0,1)$ and $G_n$ the empirical version coming from a sample $U_1, \ldots, U_n$. By the quantile transformation we can write \begin{equation*}
	\sqrt{n}\big(\hat{q}_{\cdot,n}-q_{\cdot}\big) = \sqrt{n}\big(F^{\Inv}(G_n^{\Inv}(\cdot))-F^{\Inv}(G^{\Inv}(\cdot))\big).
\end{equation*}
The process $\sqrt{n}(G_n^{\Inv}-G^{\Inv})$ converges in distribution in $(l^{\infty}[\alpha_l,\alpha_u],\Vert\cdot\Vert_{[\alpha_l, \alpha_u]})$ to $V$, see Example~19.6, \citet{vdv2000}. 

To use a functional delta-method for the hypi-semimetric, we require 
\begin{equation*}
	t_n^{-1}\big(F^{\Inv}(\alpha+t_n\,\varphi_n(\alpha))-F^{\Inv}(\alpha)\big) \to \frac{\varphi(\alpha)}{f(q_{\alpha})}
\end{equation*}
with respect to $\dif_{hypi}$ for $t_n\searrow0$ and $\varphi_n\in\ell^{\infty}[\alpha_l,\alpha_u]$, $\varphi\in\mathcal{C}[\alpha_l,\alpha_u]$ such that $\dif_{hypi}(\varphi_n,\varphi)\to0$, that is the semi-Hadamard-differentiability of the functional $\Psi_0^{\Inv}(\nu)(\alpha)=F^{\Inv}(\alpha+\nu(\alpha))$, $\nu\in\ell^{\infty}[\alpha_l,\alpha_u]$ as in Definition~B.6, \citet{buecher2014}. To also be able to deal with the bootstrap version, we directly show a slightly stronger version, the \emph{uniform} semi-Hadamard-differentiability. 

\begin{lemma}\label{lem:uniform_semi_hadam}
	Let $t_n\searrow 0\in\R$,  $\varphi_n, \nu_n\in\ell^{\infty}[\alpha_l,\alpha_u]$, $\varphi\in\mathcal{C}[\alpha_l,\alpha_u]$ such that $\dif_{hypi}(\varphi_n,\varphi)\to0$ and $\dif_{hypi}(\nu_n,0)\to0$ holds. Then the convergence \begin{equation*}
		t_n^{-1}\Big[F^{\Inv}\big\{id_{[\alpha_l,\alpha_u]}(\cdot)+\nu_n(\cdot)+t_n\,\varphi_n(\cdot)\big\}-F^{\Inv}\big\{id_{[\alpha_l,\alpha_u]}(\cdot)+\nu_n(\cdot)\big\}\Big] \to \frac{\varphi(\cdot)}{f(q_{\cdot})}
	\end{equation*}
	with respect to $\dif_{hypi}$ is true.
\end{lemma}

Using the convergence of $\sqrt{n}(G_n^{\Inv}-G^{\Inv})$, Lemma~\ref{lem:uniform_semi_hadam} and the functional delta-method, Theorem ~B.7, \citet{buecher2014}, we conclude \begin{equation*}
	\sqrt{n}\big(\hat{q}_{\cdot,n}-q_{\cdot}\big) = \sqrt{n}\Big[F^{\Inv}\big\{G_n^{\Inv}(\cdot)\big\}-F^{\Inv}\big\{G^{\Inv}(\cdot)\big\}\Big]\to \frac{V}{f(q_{\cdot})}
\end{equation*}
in $(L^{\infty}[\alpha_l,\alpha_u],\dif_{hypi})$. 
\end{proof}

\begin{proof}[Proof of Theorem~\ref{thm:main_boot_quantile}.]
	
	Given a sample $U_1, \ldots , U_n$ of independent $\mathcal{U}[0,1]$ distributed random variables and the corresponding empirical distribution function $G_n$, let $U_1^*, \ldots, U_n^*$ be a bootstrap sample with empirical distribution function $G_n^*$. Using the quantile transformation we obtain \begin{equation*}
		\sqrt{n}\big( q_{\cdot,n}^*-\hat{q}_{\cdot,n}\big) = \sqrt{n}\Big[F^{\Inv}\big\{(G_n^*)^{\Inv}(\cdot)\big\}-F^{\Inv}\big\{G_n^{\Inv}(\cdot)\big\}\Big].
	\end{equation*}
	By Theorem~3.6.2, \citet{vdvW2013}, the process $\sqrt{n}\big\{(G_n^*)^{\Inv}(\cdot)-G_n^{\Inv}(\cdot)\big\} $ converges in distribution to $V$, conditionally, almost surely, with $V$ as in Theorem~\ref{thm:conv_quantile}. 
	Now we use Theorem~3.9.13, \citet{vdvW2013}, extended to semi-metric spaces. 
	To this end we use Lemma~\ref{lem:uniform_semi_hadam}, which covers (3.9.12) in \citet{vdvW2013}. Moreover we require that $\Vert G_n^{\Inv}(\cdot)-G^{\Inv}(\cdot)\Vert \to 0$ holds almost surely, which is valid by the classical Glivenko-Cantelli-Theorem, Theorem~19.4, \citet{vdv2000}. Finally, condition (3.9.9) in \citet{vdvW2013} is valid by their Theorem~3.6.2. Thus we can indeed apply Theorem~3.9.13, \citet{vdvW2013}, which concludes the proof. 
\end{proof}

\section*{Supplementary material}

The supplementary material contains simulation results, a summary of convergence under the hypi-semimetric and technical details for the proofs. 

\section*{Acknowledgements}

Tobias Zwingmann acknowledges financial support from of the Cusanuswerk for providing a dissertation scholarship.

\bibliographystyle{chicago}
\nocite{*}
\bibliography{expectile}

\begin{thebibliography}{}

\bibitem[\protect\citeauthoryear{??}{kie}{}]{kiefer}


\bibitem[\protect\citeauthoryear{??}{mol}{}]{molchanov2005}


\bibitem[\protect\citeauthoryear{{Bellini}, {Klar}, {M\"uller}, and
  {Gianin}}{{Bellini} et~al.}{2014}]{Klar1}
{Bellini}, F., B.~{Klar}, A.~{M\"uller}, and E.~R. {Gianin} (2014).
\newblock {Generalized quantiles as risk measures.}
\newblock {\em {Insur. Math. Econ.}\/}~{\em 54}, 41--48.

\bibitem[\protect\citeauthoryear{Beyn and Rieger}{Beyn and
  Rieger}{2011}]{beyn2011}
Beyn, W.-J. and J.~Rieger (2011).
\newblock An implicit function theorem for one-sided lipschitz mappings.
\newblock {\em Set-Valued and Variational Analysis\/}~{\em 19\/}(3), 343--359.

\bibitem[\protect\citeauthoryear{{B\"ucher}, {Segers}, and
  {Volgushev}}{{B\"ucher} et~al.}{2014}]{buecher2014}
{B\"ucher}, A., J.~{Segers}, and S.~{Volgushev} (2014).
\newblock {When uniform weak convergence fails: empirical processes for
  dependence functions and residuals via epi- and hypographs.}
\newblock {\em {Ann. Stat.}\/}~{\em 42\/}(4), 1598--1634.

\bibitem[\protect\citeauthoryear{Holzmann and Klar}{Holzmann and
  Klar}{2016}]{holzmann2016}
Holzmann, H. and B.~Klar (2016).
\newblock Expectile asymptotics.
\newblock {\em Electron. J. Statist.\/}~{\em 10\/}(2), 2355--2371.

\bibitem[\protect\citeauthoryear{{Knight}}{{Knight}}{1998}]{knightboot}
{Knight}, K. (1998).
\newblock {Bootstrapping sample quantiles in non-regular cases.}
\newblock {\em {Stat. Probab. Lett.}\/}~{\em 37}, 259--267.

\bibitem[\protect\citeauthoryear{Knight}{Knight}{2002}]{Knight2002LimDistr}
Knight, K. (2002).
\newblock What are the limiting distributions of quantile estimators?
\newblock In Y.~Dodge (Ed.), {\em Statistical Data Analysis Based on the
  L1-Norm and Related Methods}, pp.\  47--65. Birkh\"auser Basel.

\bibitem[\protect\citeauthoryear{{Koenker}}{{Koenker}}{2005}]{Koenker}
{Koenker}, R. (2005).
\newblock {\em {Quantile regression.}}
\newblock Cambridge: Cambridge University Press.

\bibitem[\protect\citeauthoryear{{Newey} and {Powell}}{{Newey} and
  {Powell}}{1987}]{Newey}
{Newey}, W. and J.~{Powell} (1987).
\newblock {Asymmetric least squares estimation and testing.}
\newblock {\em {Econometrica}\/}~{\em 55\/}(4), 819--847.

\bibitem[\protect\citeauthoryear{Thomson, Bruckner, and Bruckner}{Thomson
  et~al.}{2008}]{bruckner2008}
Thomson, B., A.~Bruckner, and J.~Bruckner (2008).
\newblock {\em Real Analysis}.
\newblock www.classicalrealanalysis.com.

\bibitem[\protect\citeauthoryear{van~der Vaart}{van~der Vaart}{2000}]{vdv2000}
van~der Vaart, A. (2000).
\newblock {\em Asymptotic Statistics}.
\newblock Cambridge Series in Statistical and Probabilistic Mathematics.
  Cambridge University Press.

\bibitem[\protect\citeauthoryear{van~der Vaart and Wellner}{van~der Vaart and
  Wellner}{2013}]{vdvW2013}
van~der Vaart, A. and J.~Wellner (2013).
\newblock {\em Weak Convergence and Empirical Processes: With Applications to
  Statistics}.
\newblock Springer Series in Statistics. Springer New York.

\bibitem[\protect\citeauthoryear{{Ziegel}}{{Ziegel}}{2016}]{Ziegel1}
{Ziegel}, J.~F. (2016).
\newblock {Coherence and elicitability.}
\newblock {\em {Math. Finance}\/}~{\em 26\/}(4), 901--918.

\end{thebibliography}

\section{Supplement: Simulation Study}

In this section we illustrate the asymptotic results for the expectile process in a short simulation. 
Let $Y$ be a random variable with distribution function 
\begin{equation*}
	F(x)=\frac{9}{10}\,\int_{-\infty}^{x}\frac{1}{4\,\sqrt{2\,\pi}}\,\exp\Big(-\frac{y^2}{32}\Big)\,\dif y + \frac{1}{10}\,\ind{x\geq 1},
\end{equation*}
which is a mixture of a $\mathcal{N}(0,16)$ random variable and point-mass in $1$, so that $\E{Y}=\nicefrac{1}{10}$ and $\E{Y^2}=14.5$. We will concentrate on the weak convergence of the sup-norm of the empirical expectile process. Using equation (2.7) in \citet{Newey}, we numerically find $\mu_{\tau_0}=1$ for $\tau_0 \approx 0.6529449$, and investigate the expectile process on the interval $[0.6, 0.7]$. 

Figure~\ref{fig:paths_expectile} contains four paths of the expectile process $\sqrt{n}\big(\hat \mu_{\tau, n}-\mu_{\tau}\big)$ for samples of size $n=10^4$. 
All plotted paths seem to evolve a jump around $\tau_0$.





\begin{figure}
	\centering
	\begin{subfigure}{.45\textwidth}
		\centering
		\includegraphics[width=\linewidth]{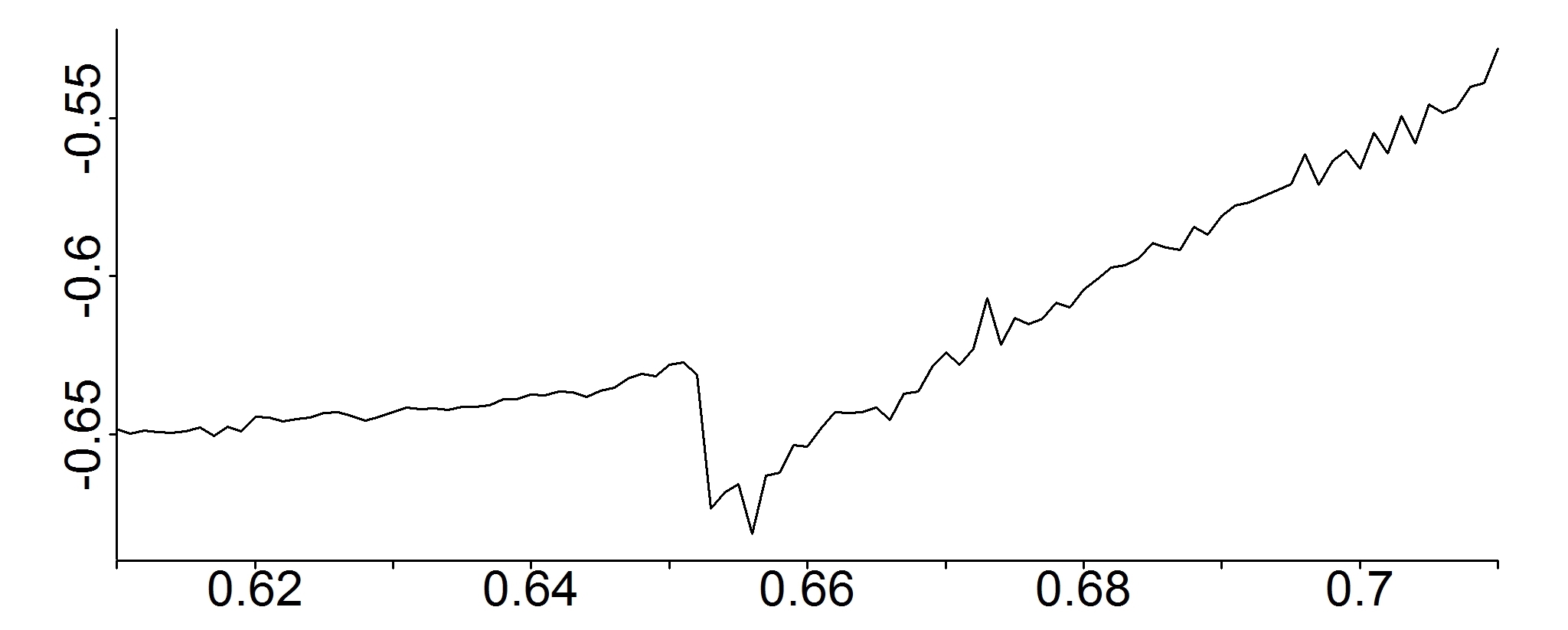}
		\captionsetup{font={small}, labelfont={bf}, width=.95\linewidth}
		\caption{Path evolving a downward jump.}
		\label{subfig:path_down_i}
	\end{subfigure}%
	\begin{subfigure}{.45\textwidth}
		\centering
		\includegraphics[width=\linewidth]{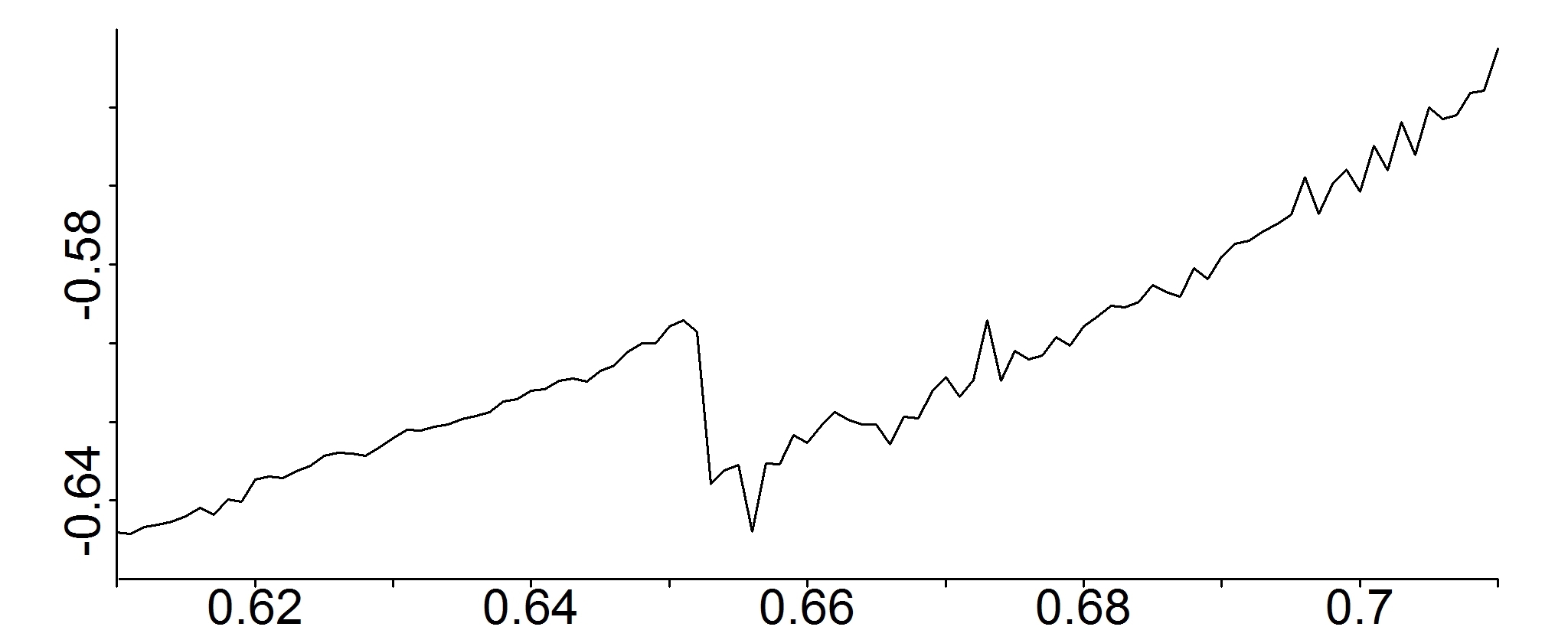}
		\captionsetup{font={small}, labelfont={bf}, width=.95\linewidth}
		\caption{Path evolving a downward jump.}
		\label{subfig:path_down_ii}
	\end{subfigure}\\
	\begin{subfigure}{.45\textwidth}
		\centering
		\includegraphics[width=\linewidth]{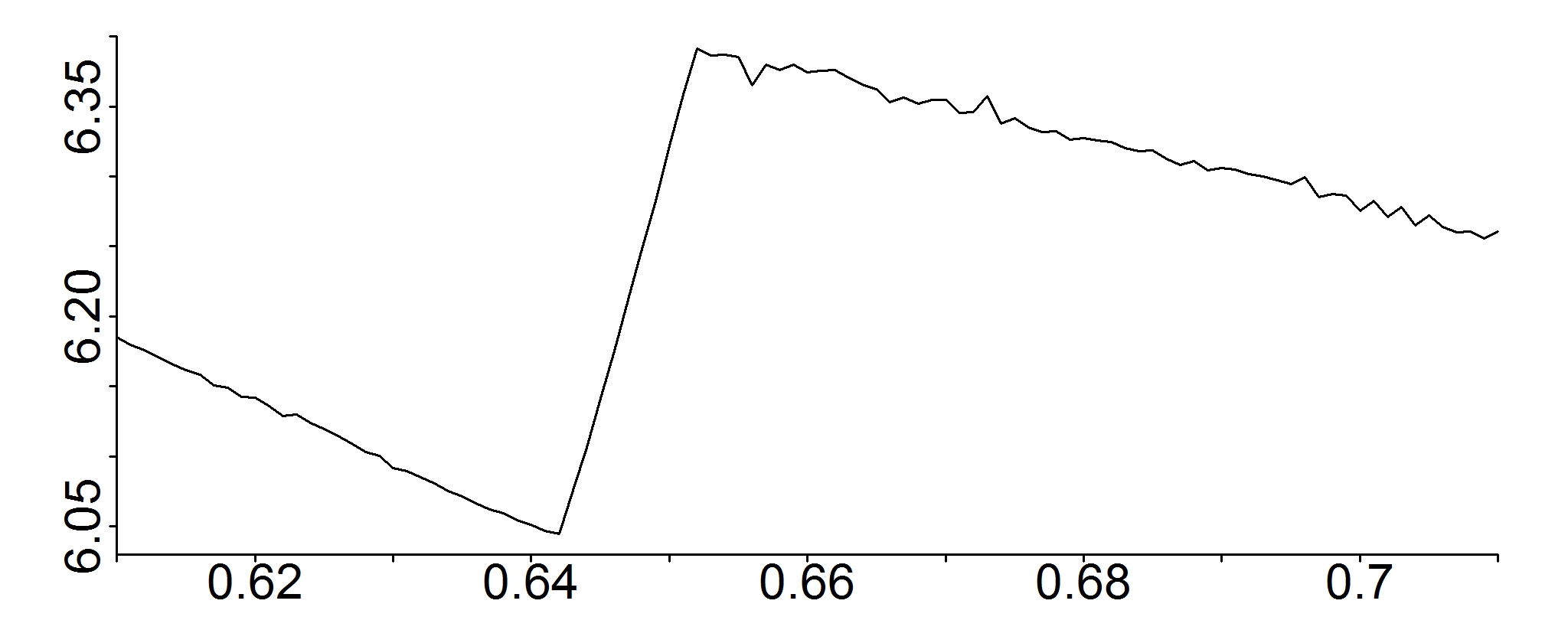}
		\captionsetup{font={small}, labelfont={bf}, width=.95\linewidth}
		\caption{Path evolving an upward jump.}
		\label{subfig:path_up_i}
	\end{subfigure}%
	\begin{subfigure}{.45\textwidth}
		\centering
		\includegraphics[width=\linewidth]{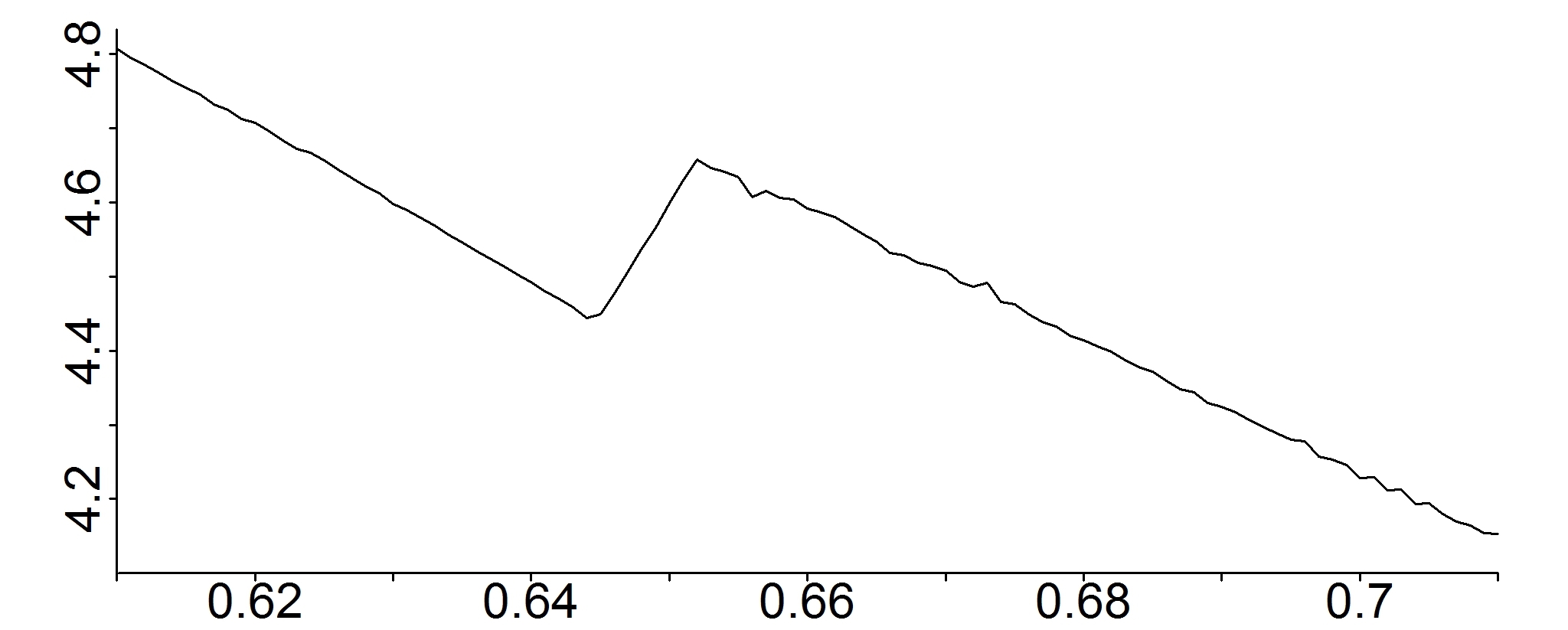}
		\captionsetup{font={small}, labelfont={bf}, width=.95\linewidth}
		\caption{Path evolving an upward jump.}
		\label{subfig:path_up_ii}
	\end{subfigure}
	\captionsetup{font={small}, labelfont={bf}, width=.95\linewidth}
	\caption[Exemplary paths of empirical expectile process.]{The pictures show simulated paths of the empirical expectile process based on $n=10^4$ observations of $Y$. If the path is negative (positive) around $\tau_0$, a downward- (upward-) jump seems to evolve. This is plausible when considering the form of the hulls of $\dot{\psi}^{\Inv}$ in the limit process.}
	\label{fig:paths_expectile}
\end{figure}

Now we investigate the distribution of the supremum norm of the expectile process on the interval $[0.6, 0.7]$. To this end, we simulate $M=10^4$ samples of sizes $n \in \{10, 10^2,10^4\}$, compute the expectile process and its supremum norm. Plots of the resulting empirical distribution function and density estimate of this statistic are contained in 
Figure~\ref{fig:emp_10_to_10000}. The distribution of the supremum distance seems to converge quickly. 

Finally, to illustrate performance of the bootstrap, Figure \ref{fig:boot_10_to_1000} displays the distribution of $M=10^4$ bootstrap samples of $\Vert\sqrt{n}\big(\mu_{\cdot, n}^*-\hat \mu_{\cdot,n}\big)\Vert$ based on a single sample of size $n \in \{10^2, 10^3, 10^4\}$ from the $n$ out of $n$ bootstrap, together with the distribution of  $\Vert\sqrt{n}\big(\hat \mu_{\cdot, n}-\mu_{\cdot}\big)\Vert$. The bootstrap distribution for $n=10^4$ is quite close to the empirical distribution.  


\begin{figure}[htb]
	\centering
	\begin{subfigure}{.5\textwidth}
		\centering
		\includegraphics[width=\linewidth]{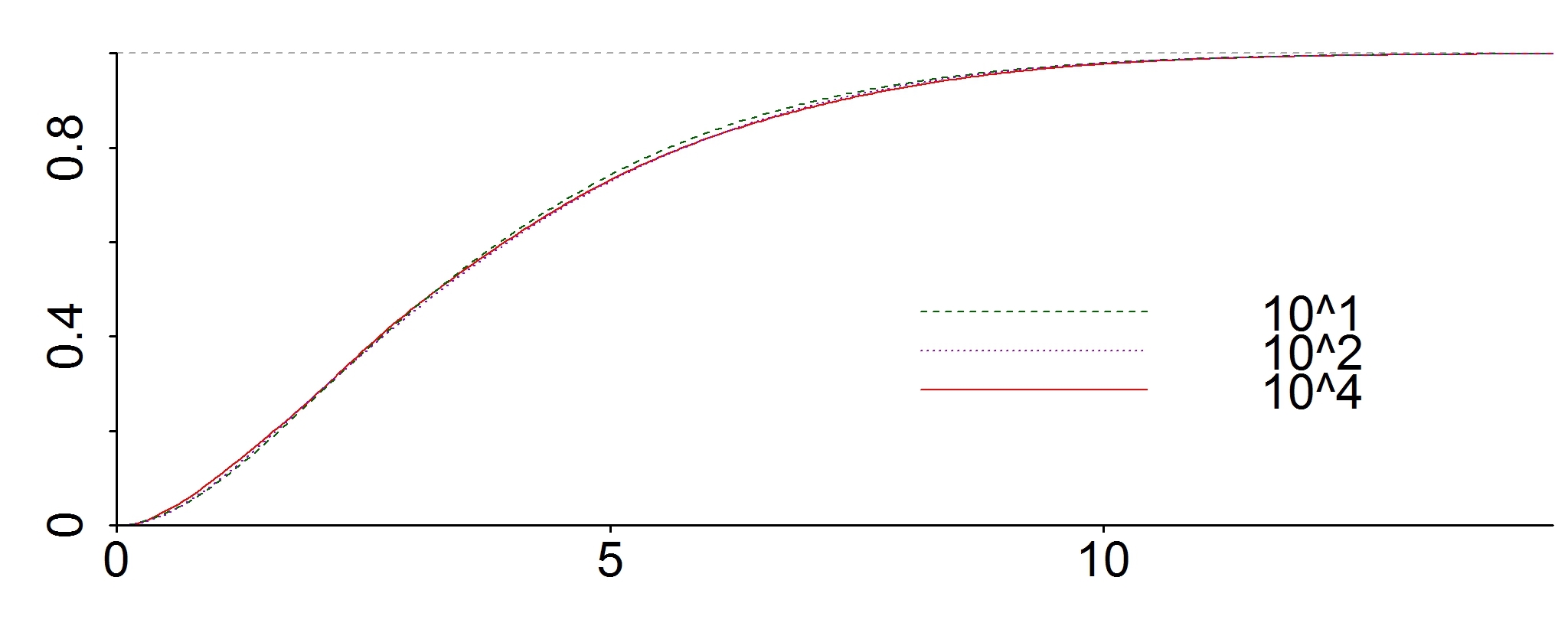}
		\captionsetup{font={small}, labelfont={bf}, width=.9\linewidth}
		\caption{Estimated cumulative distribution function.}
		\label{subfig:cdf_est_1}
	\end{subfigure}%
	\begin{subfigure}{.5\textwidth}
		\centering
		\includegraphics[width=\linewidth]{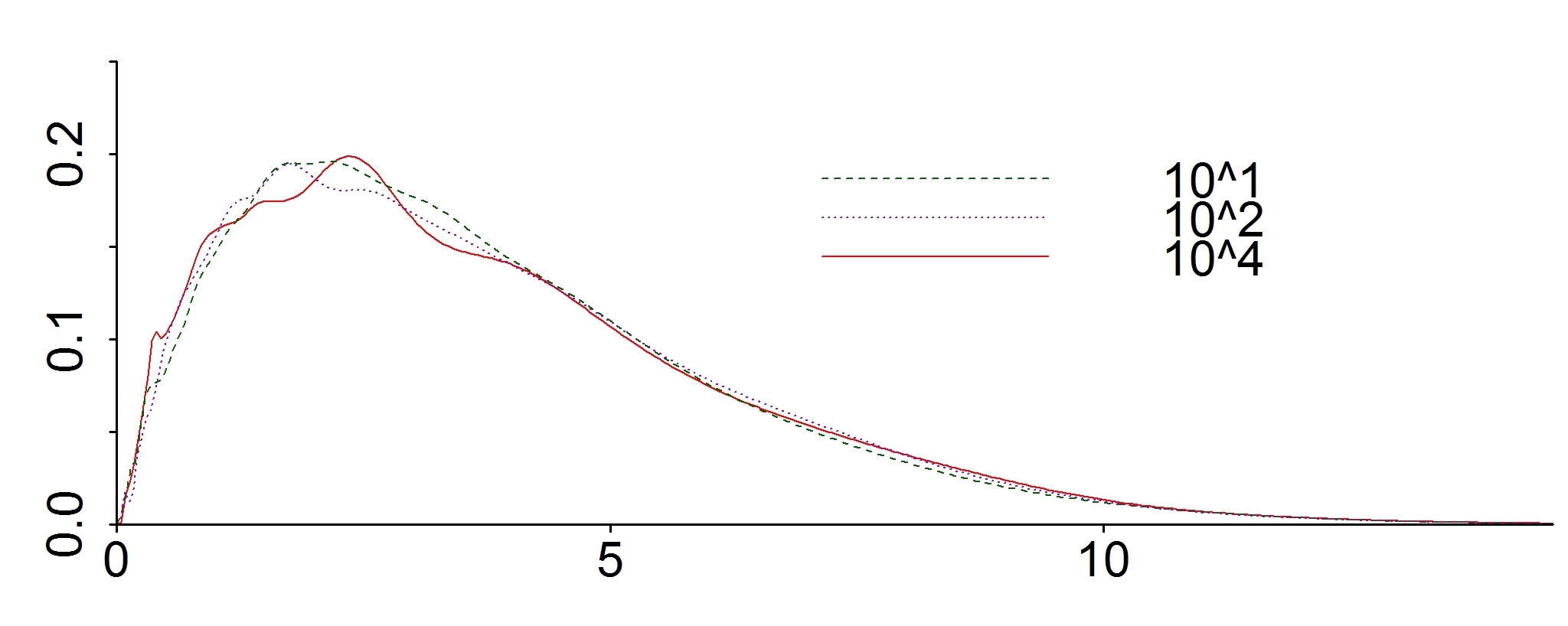}
		\captionsetup{font={small}, labelfont={bf}, width=.9\linewidth}
		\caption{Estimated density.}
		\label{subfig:dens_est_1}
	\end{subfigure}
	\captionsetup{font={small}, labelfont={bf}, width=.95\linewidth}
	\caption[Example: Estimated distributions of the supremum norm of $\sqrt{n}\big(\hat \mu_{\tau, n}-\mu_{\tau}\big)$. ]{Figure (a) shows the cumulative distribution function of the supremum norm of $\sqrt{n}\big(\hat \mu_{\tau, n}-\mu_{\tau}\big)$, based on $M=10^4$ samples of sizes $n \in \{10, 10^2, 10^4\}$, Figure (b) the corresponding density estimate. }
	\label{fig:emp_10_to_10000}
\end{figure}


\begin{figure}[htb]
	\centering
	\begin{subfigure}{.5\textwidth}
		\centering
		\includegraphics[width=\linewidth]{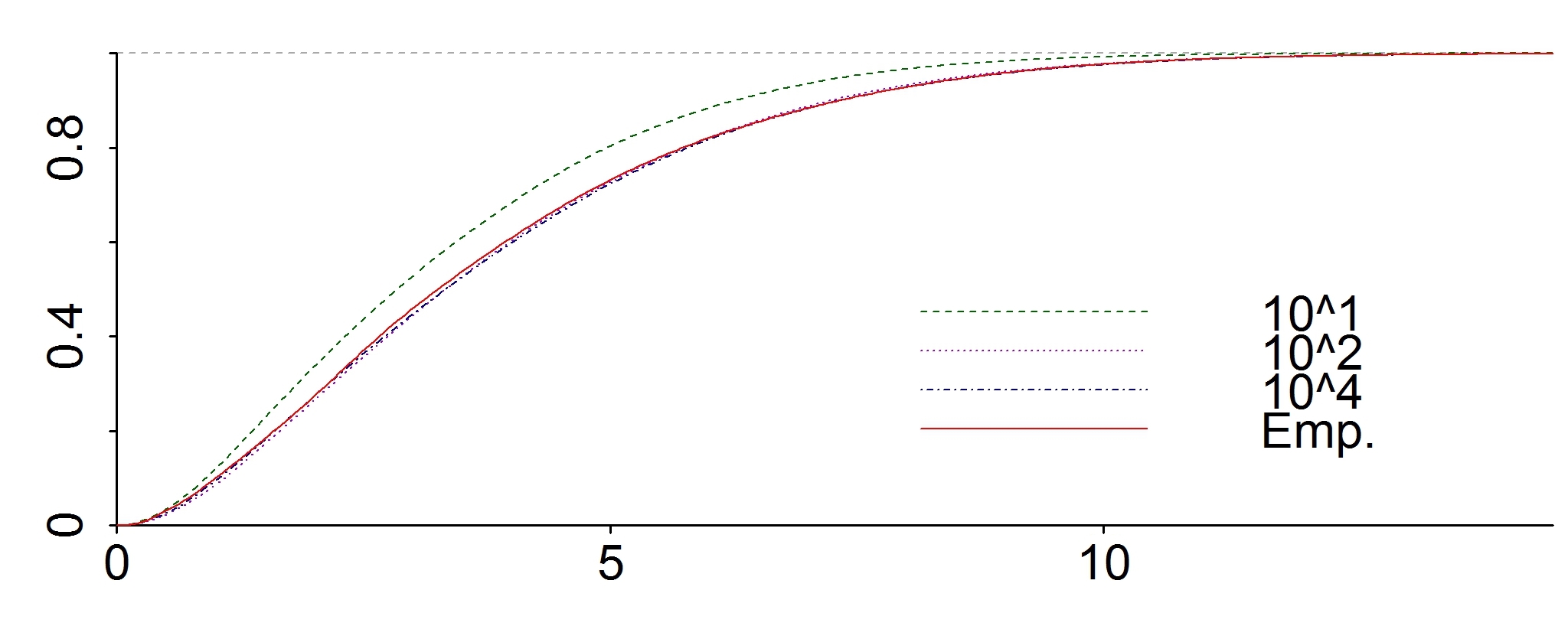}
		\captionsetup{font={small}, labelfont={bf}, width=\linewidth}
		\caption{Estimated bootstrap cumulative distribution function.}
		\label{subfig:boot_cdf_est_1}
	\end{subfigure}%
	\begin{subfigure}{.5\textwidth}
		\centering
		\includegraphics[width=\linewidth]{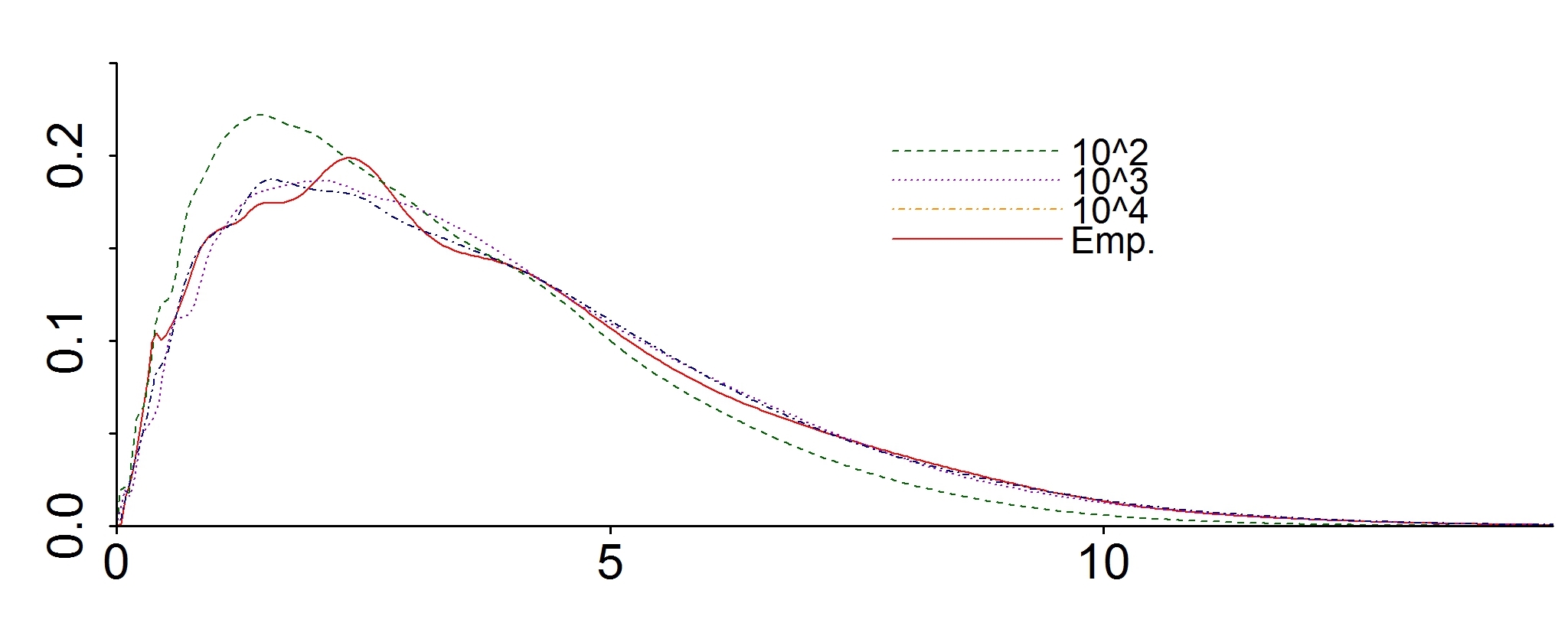}
		\captionsetup{font={small}, labelfont={bf}, width=.9\linewidth}
		\caption{Estimated bootstrap density.}
		\label{subfig:boot_dens_est_1}
	\end{subfigure}
	\captionsetup{font={small}, labelfont={bf}, width=.95\linewidth}
	\caption[Example: Estimated distributions for $\Vert\sqrt{n}\big(\mu_{\cdot, n}^*-\mu_{\cdot,n}\big)\Vert$.]{Figure (a) shows the estimated cumulative bootstrap distribution function of $\Vert\sqrt{n}\big(\hat \mu_{\cdot, n}^*-\mu_{\cdot,n}\big)\Vert$, figure (b) the estimated density thereof, obtained from $M=10^4$ estimates of this statistic. The red line indicates the estimated empirical distribution and density function, respectively, taken from $\Vert\sqrt{n}\big(\mu_{\cdot, n}-\mu_{\cdot}\big)\Vert$ for $n=10^4$. }
	\label{fig:boot_10_to_1000}
\end{figure}


\section{Supplement: weak-convergence under the hypi semimetric}
 
Let us briefly discuss the concept of hypi-convergence as introduced by \citet{buecher2014}.  
Let $(\mathbb{T},\dif )$ be a compact, separable metric space, and denote by $\ell^{\infty}(\mathbb{T})$ the space of all bounded functions $h: \mathbb{T}\to \mathbb{R}$. The lower- and upper-semicontinuous hulls of $h \in \ell^{\infty}(\mathbb{T})$ are defined by
\begin{align}\label{eq:lowerupperhulls}
\begin{split}
	h_{\wedge}(t) &= \lim_{\varepsilon \searrow 0}\,\inf\big\{h(t')\,\mid\, \dif(t,t')<\varepsilon \big\}, ~~ 
	h_{\vee}(t) = \lim_{\varepsilon \nearrow 0}\,\sup \big\{h(t')\,\mid\, \dif(t,t')<\varepsilon \big\}
\end{split}
\end{align}
and satisfy $h_{\wedge},h_{\vee}\in\ell^{\infty}(\mathbb{T})$ as well as $h_{\wedge}\leq h\leq h_{\vee}$. 
A sequence $h_n\in\ell^{\infty}(\mathbb{T})$ hypi-converges to a limit $h\in\ell^{\infty}(\mathbb{T})$, if it epi-converges to $h_\wedge$, that is, 
\begin{align}\label{eq:epi_pointwise}
\begin{split}
		\text{for all } t, t_n \in\mathbb{T} \text{ with }t_n\to t : \quad & h_{\wedge}(t)\leq \liminf_{n\to\infty} h_n(t_n)\\
		\text{for all } t \in\mathbb{T} \text{ there exist }t_n \in \mathbb{T}, \ t_n\to t : \quad & h_{\wedge}(t) =  \lim_{n\to\infty} h_n(t_n), \hfill
\end{split}
\end{align}
and if it hypo-converges to $h_\vee$, that is, 
\begin{align}\label{eq:hypo_pointwise}
\begin{split}
		\text{for all } t, t_n \in\mathbb{T} \text{ with }t_n\to t : \quad & \limsup_{n\to\infty} h_n(t_n)\leq h_{\vee}(t)\\
		\text{for all } t \in\mathbb{T} \text{ there exist }t_n \in \mathbb{T}, \ t_n\to t : \quad & \lim_{n\to\infty} h_n(t_n)= h_{\vee}(t).\hfill
\end{split}
\end{align}

The limit function $h$ is only determined in terms of its lower - and upper semicontinuous hulls. Indeed, there is a semimetric, denoted by $\dif_{hypi}$, so that the convergence in (\ref{eq:lowerupperhulls}) and (\ref{eq:hypo_pointwise}) is equivalent to $\dif_{hypi}(h_n,h) \to 0$, see \citet{buecher2014} for further details. 
To transfer the concept of weak convergence from metric to semimetric space, \citet{buecher2014} consider the space $L^{\infty}(\mathbb{T})$ of equivalence classes $[h]=\{g\in\ell^{\infty}(\mathbb{T})\,\mid\, \dif_{hypi}(h,g)=0\}$. The convergence of a sequence of random elements $(Y_n)$ in $g\in\ell^{\infty}(\mathbb{T})$ to a Borel-measurable $Y$ is defined in terms of weak convergence of $([Y_n])$ to $[Y]$ in the metric space $(L^{\infty}(\mathbb{T}), \dif_{hypi})$ in the sense of Hofmann-J\o rgensen, see \citet{vdvW2013}) and \citet{buecher2014}.


\section{Supplement: Technical Details}

We write $\En {h(Y)} = \nicefrac{1}{n} \sum_{k=1}^n h(Y_k)$, and use the abbreviation 
\begin{align*}
	\Vert \sqrt{n}\big(F_n-F)\Vert_{\mathcal{G}} = \sup_{h\in\mathcal{G}}\big|\sqrt{n}\big[\En{h(Y)}-\E{h(Y)}\big]\big|
\end{align*}
for a class of measurable functions $\mathcal{G}$. 

Recall from \citet{holzmann2016} the identity
\begin{equation}\label{eq:rep_I}
	I_{\tau}(x,F)=\tau\,\int_{x}^{\infty}\big(1-F(y)\big)\,\dif y-(1-\tau)\,\int_{-\infty}^{x}F(y)\,\dif y.
\end{equation}

\subsection{Details for the Proof of Theorem \ref{thm:main_thm}}
We start with some technical preliminaries. 
\begin{lemma}\label{lem:rep_dif_I}
We have that for $x_1,x_2\in\R$,
\begin{equation}\label{eq:rep_dif_I}
	I_{\tau}(x_1,F) -I_{\tau}(x_2,F) = (x_2-x_1)\,\big[ \tau + \big(1-2\,\tau\big)\,\int_{0}^{1}F\big(x_2+s(x_1-x_2)\big)\,\dif s\big].
	\end{equation} 
\end{lemma}

\begin{proof}
	Define $g(s)=I_{\tau}(x_2+s(x_1-x_2),F)$, $s \in [0,1]$, so that $I_{\tau}(x_1,F) -I_{\tau}(x_2,F) = g(1)-g(0)$. The map $g$ is continuous, in addition it is decreasing, if $x_1\geq x_2$, and increasing otherwise. Hence it is of bounded variation and Theorem~7.23, \citet{bruckner2008}, yields \begin{equation}\label{eq:integral_g_error}
		g(1)-g(0)= \int_{0}^{1} g'(s)\,\dif s + \mu_g(\{s\in [0,1] \mid g'(s)=\pm\infty\}),
	\end{equation}
	where $\mu_g$ is the Lebesgue-Stieltjes signed measure associated with $g$. 
%
%
	From \citet{holzmann2016}, the right- and left-sided derivatives of $I_{\tau}(x,F)$ are given by 
	\begin{equation}\label{eq:leftreightder}
		\frac{\partial^+}{\partial x} I_{\tau}(x,F)\big) = -\big(\tau+ (1-2\,\tau)F(x)\big), \qquad 
		\frac{\partial^-}{\partial x} I_{\tau}(x,F) =-\big(\tau + (1-2\,\tau)F(x-)\big).
	\end{equation}
	Both derivatives are bounded, such that $\{s\in [0,1] \mid g'(s)=\pm\infty\}=\emptyset$ in \eqref{eq:integral_g_error}, and we obtain 
	\begin{align*}
		&I_{\tau}(x_1,F) -I_{\tau}(x_2,F)  = \int_0^1 g'(s)\,\dif s\\
		= & \, (x_1-x_2)\, \int_{0}^{1}-\,\big[\tau + \big(1-2\,\tau\big)\,F\big(x_2+s(x_1-x_2)\big)\big]\dif s.
	\end{align*}
\end{proof}
\begin{lemma}\label{lem:c_bound}
We have
\begin{equation}\label{eq:estbasic}
\min\big\{\tau_l,1-\tau_u\big\} \leq \tau + (1-2\tau)s \leq 3/2,\quad \tau \in [\tau_l, \tau_u],\ s \in [0,1].
\end{equation}
\end{lemma}

\begin{proof}
For the lower bound, 
\begin{equation*}
	\tau + (1-2\,\tau)s  \left\{\begin{matrix}
	= \nicefrac{1}{2},\hfill &&\hspace{-7pt}\mbox{if } \tau=\nicefrac{1}{2}\\
	\geq \tau , \hfill &&\hspace{-7pt}\mbox{if } \tau<\nicefrac{1}{2}\\
	\geq 1-\tau, &&\hspace{-7pt}\mbox{if } \tau >\nicefrac{1}{2}
	\end{matrix}\right\} \geq \min\big\{\nicefrac{1}{2},\tau_l,1-\tau_u\big\} = \min\big\{\tau_l,1-\tau_u\big\}.
	\end{equation*}
	The upper bound is proved similarly. 
\end{proof}
Next, we discuss Lipschitz-properties of relevant maps. 
\begin{lemma}\label{lem:lipschitz}
For any $x_1, x_2, y \in \R$ and $\tau \in [\tau_l, \tau_u]$, 
\begin{equation}\label{eq:lip1}
		\big|I_{\tau}(x_1,y)-I_{\tau}(x_2,y)\big|\leq \big|x_2-x_1\big| 
\end{equation}
Further, for any $\tau,\tau'\in[\tau_l,\tau_u]$ and $x,y\in \R$,
\begin{equation}\label{eq:lip2}
		\big|I_{\tau}(x,y)-I_{\tau'}(x,y)\big|\leq \big|\tau-\tau'\big|\,\big(|x|+|y|\big)
\end{equation}
Finally, the map $\tau\mapsto\mu_{\tau}$, $\tau \in [\tau_l, \tau_u]$, is Lipschitz-continuous. 
\end{lemma}
\begin{proof}
To show (\ref{eq:lip1}), for $x_1 \leq x_2$, 
\begin{align*}
		\big|I_{\tau}(x_1,y)-I_{\tau}(x_2,y)\big|&=\big|\big(x_2-x_1\big)\,\big(\tau \ind{y>x_1}+(1-\tau)\ind{y\leq x_2}\big)\big|\\
		& \leq \big|x_2-x_1\big| . 
	\end{align*} 
As for (\ref{eq:lip2}), 
\begin{align*}
		\big|I_{\tau}(x,y)-I_{\tau'}(x,y)\big| &= \big|\tau-\tau'\big|\big|(y-x)1_{y\geq x}+(x-y)1_{y<x}\big|\\
		&\leq \big|\tau-\tau'\big|\,\big(|x|+|y|\big).
	\end{align*}
For the Lipschitz-continuity of $\mu_\cdot$, we use Corollary~1 of \citet{beyn2011} for the function $x\mapsto I_{\tau}(x,F)$, $x \in B_R(\mu_{\tau})$, for appropriately chosen $R>0$, where $B_R(\mu_{\tau})$ is the open ball around $\mu_{\tau}$ with radius $R$.  
We observe that 
 \begin{enumerate}
		\item $x\mapsto I_{\tau}(x,F)$ is continuous for any $\tau\in[\tau_l,\tau_u]$, which is immediate from (\ref{eq:lip1}), 
		\item $x\mapsto I_{\tau}(x,F)$ fulfils \begin{equation*}
		\big(I_{\tau}(x_1,F)-I_{\tau}(x_2,F)\big)\,\big(x_1-x_2\big) \leq -a\,(x_1-x_2)^2
		\end{equation*}
		with $a = \min\{\tau_l, 1-\tau_u\} > 0$. This is clear from (\ref{eq:rep_dif_I}) and (\ref{eq:estbasic}).  
	\end{enumerate}
	Let $\tau,\,\tau'\in[\tau_l,\tau_u]$, and set $z=I_{\tau}(\mu_{\tau'},F)$. Using (ii) above yields \begin{align*}
		\frac{1}{a}\,\big|z\big|= \frac{1}{a}\,\big|I_{\tau}(\mu_{\tau'},F)-I_{\tau}(\mu_{\tau},F)\big|\leq\big|\mu_{\tau}-\mu_{\tau'}\big|\leq \mu_{\tau_u}-\mu_{\tau_l}.
	\end{align*}
Choosing $R=\mu_{\tau_u}-\mu_{\tau_l}$ gives $[\mu_{\tau_l}, \mu_{\tau_u}]\subset B_R(\mu_{\tau})\subset [\mu_{\tau_l}-R, \mu_{\tau_u}+R]$, $\tau \in [\tau_l, \tau_u]$, since $\mu_{\tau}\in[\mu_{\tau_l}, \mu_{\tau_u}]$, and Corollary~1, \citet{beyn2011}, now gives a $\bar{x}\in B_R(\mu_{\tau})$ with $I_{\tau}(\bar{x},F)=z$ and \begin{equation*}
		\big|\mu_{\tau}-\bar{x}\big|\leq \frac{1}{a}\, \big|z\big|.
	\end{equation*}
	Since $x\mapsto I_{\tau}(x,F)$ is strictly decreasing and $[\mu_{\tau_l}, \mu_{\tau_u}]\subset B_R(\mu_{\tau})$ as well as $I_{\tau}(\bar{x},F)=z = I_{\tau}(\mu_{\tau'},F)$, we obtain $\bar{x}=\mu_{\tau'}$. We conclude that
	\begin{align}\label{eq:lip_const_mu}
		 \big|\mu_{\tau}-\mu_{\tau'}\big|\leq &\,\frac{1}{a}\, \big|I_{\tau}(\mu_{\tau'},F)\big| = \frac{1}{a}\, \big|\E{I_{\tau}(\mu_{\tau'},Y)- I_{\tau'}(\mu_{\tau'},Y)}\big|\notag\\
		\leq &\, \big|\tau-\tau'\big|\,\frac{|\mu_{\tau'}|+\E{|Y|}}{a}\leq \big|\tau-\tau'\big|\,\frac{|\mu_{\tau_u}|\vee |\mu_{\tau_l}|+\E{|Y|}}{a},
	\end{align}
where we used (\ref{eq:lip2}). 
\end{proof}

{\sl Details for Step 1.}

\begin{proof}[Proof of Lemma \ref{lem:expecexapnsion}.]
	
{\sl Proof of (\ref{convergenceprocess}).}
		
By Lemma \ref{lem:lipschitz} the function class \begin{equation*}
		\cF=\big\{y\mapsto -I_{\tau}(\mu_{\tau},y) \,\mid\, \tau\in[\tau_l,\tau_u] \big\}
	\end{equation*}
is Lipschitz-continuous in the parameter $\tau$ for given $y$, and the Lipschitz constant (which depends on $y$) is square-integrable under $F$. Indeed, the triangle inequality first gives \begin{equation*}
	\big|I_{\tau}(\mu_{\tau},y)-I_{\tau'}(\mu_{\tau'},y) \big| \leq \big|I_{\tau}(\mu_{\tau},y)-I_{\tau'}(\mu_{\tau},y) \big|+\big|I_{\tau'}(\mu_{\tau},y)-I_{\tau'}(\mu_{\tau'},y) \big|.
\end{equation*}
Using \eqref{eq:lip2} the first summand on the right fulfils \begin{align*}
	\big|I_{\tau}(\mu_{\tau},y)-I_{\tau'}(\mu_{\tau},y) \big|\leq |\tau-\tau'|\,(|\mu_{\tau_l}|\vee |\mu_{\tau_u}|+|y|),
\end{align*}
and the second is bounded by \begin{align*}
	\big|I_{\tau'}(\mu_{\tau},y)-I_{\tau'}(\mu_{\tau'},y) \big|\leq |\mu_{\tau}-\mu_{\tau'}|\, \leq |\tau-\tau'|\,\frac{|\mu_{\tau_u}|\vee |\mu_{\tau_l}|+\E{|Y|}}{a},
\end{align*}
utilizing \eqref{eq:lip1} and \eqref{eq:lip_const_mu}. Thus \begin{equation}\label{eq:lip_cF}
	\big|I_{\tau}(\mu_{\tau},y)-I_{\tau'}(\mu_{\tau'},y) \big| \leq |\tau-\tau'|\,(C+|y|)
\end{equation}
for some constant $C\geq 1$. By example~19.7 in combination with Theorem~19.5 in \citet{vdv2000}, $\cF$ is a Donsker class, so that $\sqrt{n} \big(\psi_n(\mu_\cdot) - \psi_0(\mu_\cdot)  \big)$ converges to the process $Z$. The same reasoning as in Theorem~8, \citet{holzmann2016}, then shows continuity of the sample paths of $Z$ with respect to the Euclidean distance on $[\tau_l,\tau_u]$.

\medskip

{\sl Proof of (\ref{eq:maximalprocess}).} 

Setting 
\begin{equation*}
	\mathcal{F}_{\delta_n} =\big\{y\mapsto I_{\tau}\big(\mu_{\tau}+x,y\big)-I_{\tau}\big(\mu_{\tau},y\big) \,\mid \, |x|\leq \delta_n,\, \tau\in[\tau_l,\tau_u] \big\}
\end{equation*}	
we estimate that
\begin{align*}
	 \sup_{\| \varphi\|_{[\tau_l, \tau_u]} \leq \delta_n}\, \sup_{\tau \in [\tau_l, \tau_u]}\, \sqrt{n}\, \big|\psi_n(\mu_\cdot + \varphi)(\tau) - \psi_0(\mu_\cdot + \varphi)(\tau) - \big[ \psi_n(\mu_\cdot)(\tau) - \psi_0(\mu_\cdot)(\tau) \big] \big|
\end{align*}
is smaller than $ \Vert \sqrt{n}\big(F_n-F)\Vert_{\mathcal{F}_{\delta_n}}$. From the triangle inequality, for any $\tau,\tau'\in[\tau_l,\tau_u]$ and $x, x' \in [-\delta_1, \delta_1]$ we first obtain
\begin{align*}
	&\big|I_{\tau}\big(\mu_{\tau}+x,y\big)-I_{\tau}\big(\mu_{\tau},y\big)-\big(I_{\tau'}\big(\mu_{\tau'}+x',y\big)-I_{\tau'}\big(\mu_{\tau'},y\big)\big)\big| \\
	\leq &\big|I_{\tau}\big(\mu_{\tau}+x,y\big)-I_{\tau'}\big(\mu_{\tau'}+x',y\big)\big|+\big|I_{\tau}\big(\mu_{\tau},y\big)- I_{\tau'}\big(\mu_{\tau'},y\big)\big|,
\end{align*}
where the second term was discussed above and the first can be handled likewise to conclude \begin{align}\label{eq:lip_x_plus_mu}
	&\big|I_{\tau}\big(\mu_{\tau}+x,y\big)-I_{\tau'}\big(\mu_{\tau'}+x',y\big)\big|\leq \big(|\tau-\tau'|+|x-x'|\big)\,\big(C+\delta_1+|y|\big)
\end{align}
with the same $C$ as above. Hence \begin{align*}
	&\big|I_{\tau}\big(\mu_{\tau}+x,y\big)-I_{\tau}\big(\mu_{\tau},y\big)-\big(I_{\tau'}\big(\mu_{\tau'}+x',y\big)-I_{\tau'}\big(\mu_{\tau'},y\big)\big)\big|\\
	 &\leq m(y)\,\big(|\tau-\tau'|+|x-x'|\big)
\end{align*}
with Lipschitz-constant \begin{equation*}
	m(y)=2\,C+\delta_1 + 2\,|y|,
\end{equation*}
which is square-integrable by assumption on $F$. By example~19.7 in \citet{vdv2000} the bracketing number $N_{[\,]}\big(\epsilon, \mathcal{F}_{\delta_1}, L_2(F)\big)$ of $\cF_{\delta_1}$ is of order $\epsilon^{-2}$, so that for the bracketing integral
\[ J_{[\,]} \big(\epsilon_n,\mathcal{F}_{\delta_n},L_2(F)\big) \leq J_{[\,]} \big(\epsilon_n,\mathcal{F}_{\delta_1},L_2(F)\big) \to 0 \quad \text{ as } \epsilon_n \to 0.\]
From (\ref{eq:lip1}), the class $\cF_{\delta_n}$ has envelope $\delta_n$, and hence using Corollary~19.35 in \citet{vdv2000}, we obtain
\begin{align}\label{eq:bracket_int_to_zero}
\E{\Vert \sqrt{n}\big(F_n-F)\Vert_{\mathcal{F}_{\delta_n}} }\leq J_{[\,]} \big(\delta_n,\mathcal{F}_{\delta_n},L_2(F)\big) \to 0.
\end{align}
An application of the Markov inequality ends the proof of (\ref{eq:maximalprocess}). 
\end{proof}

{\sl Details for Step 2.}

\begin{proof}[Proof of Lemma \ref{lem:invertable}.]
Given $\tau \in [\tau_l, \tau_u]$, by \eqref{eq:rep_dif_I} and the lower bound in \eqref{eq:estbasic}, the function 
$x\mapsto I_{\tau}(x,F)$ is strictly decreasing, and its image is all of $\R$. Hence, for any $z \in \R$ there is a unique $x$ satisfying $I_{\tau}(x,F) = z$, which shows that $\psi_0$ is invertible. 

Next for fixed $\varphi\in\ell^{\infty}[\tau_l,\tau_u]$ the preimage $\big((I_{\tau}(\cdot,F)\big)^{\Inv}([-\Vert\varphi\Vert,\Vert\varphi\Vert])$ is by monotonicity an interval $[L_{\tau},U_{\tau}]$, $|L_{\tau}|,|U_{\tau}|<\infty$. By \eqref{eq:rep_I},  \begin{equation*}
	I_{\tau}(x,F)=\tau\,\Big\{\int_{x}^{\infty}\big(1-F(y)\big)\,\dif y+\int_{-\infty}^{x}F(y)\,\dif y\Big\}-\int_{-\infty}^{x}F(y)\,\dif y,
\end{equation*}
thus the map $\tau\mapsto I_{\tau}(x,F)$ is increasing, showing $L_{\tau'}\leq L_{\tau}$ and $U_{\tau'}\leq U_{\tau}$ for $\tau\geq \tau'$. Hence the solution of $z=I_{\tau}(x,F)$ for $z\in[-\Vert\varphi\Vert,\Vert\varphi\Vert]$ lies in $[L_{\tau_l},U_{\tau_u}]$, which means that $\psi_0^{\Inv}(\varphi)$ is bounded. 
\end{proof}

Before we turn to the proof of Lemma \ref{thm:semi_hadamard_diff}, we require the following technical assertions about $\psi_0^{\Inv}$. 

\begin{lemma}\label{lem:rep_diff_psiinv}
Given $t>0$ and $\nu\in\ell^{\infty}[\tau_l,\tau_u]$, we have that

\begin{align}\label{eq:represpsiinv}
	t^{-1}\big(&\psi_0^{\Inv}(t\,\nu)-\psi_0^{\Inv}(0)\big)(\tau) \notag\\
	&= \nu(\tau)\Big\{\tau+(1-2\,\tau)\,\int_{0}^{1}F\big(\mu_{\tau}+s\big(\psi_0^{\Inv}(t\,\nu)(\tau)-\mu_{\tau}\big)\big)\,\dif s\Big\}^{-1}.
\end{align}
In particular, if $\nu_n\in\ell^{\infty}[\tau_l, \tau_u]$ with $\|\nu_n\|\to 0$, then $\|\psi_0^{\Inv}(\nu_n)(\cdot) - \mu_{\cdot}\| \to 0$, so that for any $\tau \in [\tau_l, \tau_u]$ and $\tau_n \to \tau$,  
\begin{equation}\label{eq:convergencepointwise}
\psi_0^{\Inv}(\nu_n)(\tau_n)-\mu_{\tau_n} \to 0.
\end{equation}
\end{lemma}

\begin{proof}
For the first statement, given $\rho\in\ell^{\infty}[\tau_l,\tau_u]$ it follows from (\ref{eq:rep_dif_I}) that
\begin{align*}
		\psi_0\big(\rho\big)(\tau) &= \psi_0\big(\mu_{\cdot} + \big(\rho -\mu_{\cdot}\big)\big)(\tau) - \psi_0\big(\mu_{\cdot}\big)(\tau)\\
		&= -\Big(I_{\tau}\big(\mu_{\tau}+\big(\rho(\tau) -\mu_{\tau}\big),F\big) - I_{\tau}(\mu_{\tau},F)\Big)\\
		&=\big(\rho(\tau)-\mu_{\tau}\big)\,\Big\{\tau + (1-2\,\tau)\int_{0}^{1}F\big(\mu_{\tau} + s(\rho(\tau)-\mu_{\tau})\big)\,\dif s\Big\}.
	\end{align*}
The integral on the right hand side is bounded away from zero by (\ref{eq:estbasic}) and thus choosing $\rho=\psi_0^{\Inv}(t\,\nu)$, observing $\mu_\tau = \psi_0^{\Inv}(0)(\tau)$ and reorganising the above equation leads to \begin{align}
	&t^{-1}\Big(\psi_0^{\Inv}(0+t\,\nu)-\psi_0^{\Inv}(0)\Big)(\tau)\notag\\
	&= t^{-1}\psi_0\big(\psi_0^{\Inv}\big(t\,\nu\big)\big)(\tau)\,\Big\{\tau+(1-2\,\tau)\int_{0}^{1}F\big(\mu_{\tau}+s\big(\psi_0^{\Inv}(t\,\nu)(\tau)-\mu_{\tau}\big)\big)\,\dif s\Big\}^{-1}\notag\\
	&=\nu(\tau) \,\Bigg\{\tau+(1-2\,\tau)\int_{0}^{1}F\big(\mu_{\tau}+s\big(\psi_0^{\Inv}(t\,\nu)(\tau)-\mu_{\tau}\big)\big)\,\dif s\Bigg\}^{-1}.\label{eq:reprinversephi}
\end{align}

In the sequel, given $\varphi \in\ell^{\infty}[\tau_l,\tau_u]$ let us introduce the function
\begin{equation}\label{eq:notationfct}
	c_{\varphi}(\tau) = \tau + (1-2\,\tau)\,\int_{0}^{1}F\big(\mu_{\tau} + s\,\varphi(\tau)\big)\,\dif s.
\end{equation}
By (\ref{eq:estbasic}), $\min\big\{\tau_l,1-\tau_u\big\}  \leq  c_{\varphi} \leq 3/2$ holds uniformly for any $\varphi$. 

Now, for the second part, set $\varphi_n= \big(\psi_0^{\Inv}(\nu_n)(\cdot)-\mu_{\cdot}\big)$, then (\ref{eq:represpsiinv}) gives with $t=1$ 
\begin{align*}
		\big\Vert \psi_0^{\Inv}(\nu_n)(\cdot)-\mu_{\cdot}\big\Vert \leq \big\Vert \nu_n\big\Vert \,\Vert c_{\varphi_n}\Vert^{-1} \leq \big\Vert \nu_n\big\Vert \big(\min\big\{\tau_l,1-\tau_u\big\}\big)^{-1}  \to 0.
	\end{align*}
Then (\ref{eq:convergencepointwise}) follows by continuity of $\mu_\cdot$, see Lemma \ref{lem:lipschitz}. 
\end{proof}

\begin{proof}[Proof of Lemma \ref{thm:semi_hadamard_diff}.]
	Let $t_n\to 0$, $t_n>0$, $(\varphi_n)_n\subset \ell^{\infty}[\tau_l, \tau_u]$ with $\varphi_n\to\varphi \in \mathcal{C}[\tau_l, \tau_u]$ with respect to $d_{hypi}$ and thus uniformly by Proposition 2.1 in \citet{buecher2014}. From (\ref{eq:represpsiinv}), and using the notation (\ref{eq:notationfct}) we can write 
	\begin{equation*}
		t_n^{-1}\big(\psi_0^{\Inv}(t_n\,\varphi_n)-\psi_0^{\Inv}(0)\big) = \varphi_n/c_{\kappa_n}, \qquad \kappa_n(\tau)= \psi_0^{\Inv}(t_n\,\varphi_n)(\tau)-\mu_{\tau}
	\end{equation*}
and we need to show that
\begin{equation}\label{eq:hypiconvtoshow}
	\varphi_n/c_{\kappa_n} \to \dot{\psi}^{\Inv}(\varphi) = \varphi/ c_0
	\end{equation}
	with respect to $d_{hypi}$, where $c_0$ is as in (\ref{eq:notationfct}) with $\varphi=0$.  
	
	Now, since $\varphi_n \to \varphi$ uniformly and $\varphi$ is continuous, to obtain (\ref{eq:hypiconvtoshow}) if suffices by Lemma~\ref{lem:diverses}, i) and iii), to show that $c_{\kappa_n} \to c_0$ under $d_{hypi}$. To this end, by Lemma~A.4, \citet{buecher2014} and Lemma~\ref{lem:diverses}, iii), it suffices to show that under $d_{hypi}$
\begin{equation}\label{eq:hypiremains}
	h_n(\tau) = \int_{0}^{1}F\big(\mu_{\tau}+s\big(\psi_0^{\Inv}(t_n\,\varphi_n)(\tau)-\mu_{\tau}\big)\big)\,\dif s \to F(\mu_\tau) = h(\tau),
\end{equation}
for which we shall use Corollary~A.7 in \citet{buecher2014}. Let
\[ \mathbb{T}=[\tau_l ,\tau_u], \qquad \mathcal{S} = \mathbb{T}\setminus \big\{\tau\in[\tau_l,\tau_u] \,\mid\, F \text{ is not continuous in } \mu_{\tau} \big\},\]
so that $\mathcal{S}$ is dense in $\mathbb{T}$ and $h\arrowvert_{\mathcal{S}}$ is continuous. Using the notation from \citet{buecher2014}, Appendix A.2, we have that 
\begin{equation}\label{eq:does_not_work_for_quantile}
	\big(h\arrowvert_{\mathcal{S}}\big)_{\wedge}^{\mathcal{S}:\mathbb{T}}=h_{\wedge}=F(\mu_{\cdot}-) \quad\text{and}\quad \big(h\arrowvert_{\mathcal{S}}\big)_{\vee}^{\mathcal{S}:\mathbb{T}}=h_{\vee}=h,
	\end{equation}
where the first equalities follow from the discussion in \citet{buecher2014}, Appendix A.2, and the second equalities from 
Lemma~\ref{lem:diverses}, ii) below, and where $F(x-) = \lim_{t\uparrow x} F(t)$ denotes the left-sided limit of $F$ at $x$. If we show that 
 \begin{enumerate}
		\item for all $\tau\in[\tau_l,\tau_u]$ with $\tau_n\to\tau$ it holds that $\liminf_n \,h_n(\tau_n)\geq F(\mu_{\tau}-)$ and
		\item for all $\tau\in[\tau_l,\tau_u]$ with $\tau_n\to\tau$ it holds that $\limsup_n \,h_n(\tau_n)\leq F(\mu_{\tau})$,
	\end{enumerate}
	Corollary~A.7 in \citet{buecher2014} implies (\ref{eq:hypiremains}), which concludes the proof of the convergence in (\ref{eq:hypiconvtoshow}). 
	
To this end, concerning (i), we compute that
\begin{align*}
	F(\mu_{\tau}-)&\leq\int_{0}^{1}\liminf_{n} F\big(\mu_{\tau_n}+s\,\big(\psi_0^{\Inv}(t_n\,\varphi_n)(\tau_n)-\mu_{\tau_n}\big)\big)\,\dif s\\
	&\leq \liminf_{n}\int_{0}^{1} F\big(\mu_{\tau_n}+s\,\big(\psi_0^{\Inv}(t_n\,\varphi_n)(\tau_n)-\mu_{\tau_n}\big)\big)\,\dif s = \liminf_n \,h_n(\tau_n),
	\end{align*}
where the first inequality follows from (\ref{eq:convergencepointwise}) and the fact that $	F(\mu_{\tau}-) \leq F(\mu_{\tau})$, and the second inequality follows from Fatou's lemma. For (ii) we argue analogously
\begin{align*}
	F(\mu_{\tau})&\geq \int_{0}^{1}\limsup_{n} F\big(\mu_{\tau_n}+s\,\big(\psi_0^{\Inv}(t_n\,\varphi_n)(\tau_n)-\mu_{\tau_n}\big)\big)\,\dif s\\
	&\geq \limsup_{n}\int_{0}^{1} F\big(\mu_{\tau_n}+s\,\big(\psi_0^{\Inv}(t_n\,\varphi_n)(\tau_n)-\mu_{\tau_n}\big)\big)\,\dif s = \limsup_n \,h_n(\tau_n).
	\end{align*}
This concludes the proof of the lemma. 	
\end{proof}

\subsection{Details for the proof of Theorem \ref{thm:main_thm_boot}}
We let $\psi_n^*(\varphi)(\tau) = - I_\tau\big(\varphi(\tau), F_n^* \big)$, $\varphi \in \ell^{\infty}[\tau_l, \tau_u]$, and denote by $\prob_n^*$ the conditional law of $Y_1^*, \ldots, Y_n^*$ given $Y_1, \ldots, Y_n$, and by $\mathbb{E}_n^*$ expectation under this conditional law.   
\begin{lemma}\label{lem:bootstrapfirststep}
We have, almost surely, conditionally on $Y_1, Y_2, \ldots $, the following statements. \begin{enumerate}
	\item If $\E{|Y|}<\infty$, then
	\begin{equation}\label{eq:boot_consist}
	\sup_{\tau \in [\tau_l, \tau_u]} \big|\mu_{\tau, n}^*-\hat{\mu}_{\tau,n}\big| = o_{\prob_n^*}(1).
	\end{equation}
\end{enumerate} 
Now assume $\E{Y^2} < \infty$. 

\begin{enumerate}[resume*]
	
	\item Weakly in $(\ell^\infty[\tau_l, \tau_u],\Vert\cdot\Vert)$ it holds that
	\begin{equation}\label{eq:condconv}
		\sqrt{n}\big(\psi_n^*(\hat{\mu}_{\cdot,n})-\psi_n(\hat{\mu}_{\cdot,n})\big) \to Z
	\end{equation}
	with $Z$ as in Theorem~\ref{thm:main_thm}.
	
	\item For every sequence $\delta_n\to 0$ it holds that 
	\begin{align}\label{eq:developboot}
		\begin{split}
		& \sup_{\| \varphi\| \leq \delta_n}\, \sup_{\tau \in [\tau_l, \tau_u]}\, \sqrt{n}\, \big|\psi_n^*(\hat{\mu}_{\cdot,n} + \varphi)(\tau) - \psi_n(\hat{\mu}_{\cdot,n} + \varphi)(\tau) \\
		& \qquad \qquad \qquad \qquad \qquad- \big[ \psi_n^*(\hat{\mu}_{\cdot,n})(\tau) - \psi_n(\hat{\mu}_{\cdot,n}+)(\tau) \big] 	\big|=  o_{\prob_n^*}(1).
		\end{split}
	\end{align}
	
	\item Weakly in $(\ell^\infty[\tau_l, \tau_u],\Vert\cdot\Vert)$ we have that
	\begin{align}\label{eq:conv_psi_n_star}
		\sqrt{n}\big(\psi_n(\mu_{\cdot,n}^*)-\psi_n(\hat{\mu}_{\cdot,n})\big) \to Z.
	\end{align}
\end{enumerate}

\end{lemma}

\begin{proof}[{Proof of Lemma \ref{lem:bootstrapfirststep}.}]

First consider (\ref{eq:boot_consist}). We start with individual consistency, the proof of which is inspired by Lemma~5.10, \citet{vdv2000}. 
	Since $I_{\tau}(\mu^*_{\tau,n},F_n^*)=0$ and $x\mapsto I_{\tau}(x,F_n^*)$ is strictly decreasing, for any $\varepsilon,\eta_l,\eta_u>0$ the inequality $I_{\tau}(\hat{\mu}_{\tau,n}-\varepsilon,F_n^*)>\eta_l$ implies $\mu^*_{\tau,n}>\hat{\mu}_{\tau,n}-\varepsilon$ and from $I_{\tau}(\hat{\mu}_{\tau,n}+\varepsilon,F_n^*)<-\eta_u$ it follows that $\mu^*_{\tau,n}<\hat{\mu}_{\tau,n}+\varepsilon$. Thus \begin{equation*}
		\prob_n^*\big(I_{\tau}(\hat{\mu}_{\tau,n}-\varepsilon,F_n^*)>\eta_l,\, I_{\tau}(\hat{\mu}_{\tau,n}+\varepsilon,F_n^*)<-\eta_u\big) \leq \prob_n^*\big(\hat{\mu}_{\tau,n}-\varepsilon<\mu^*_{\tau,n}<\hat{\mu}_{\tau,n}+\varepsilon\big)
	\end{equation*}
	and it suffices to show almost sure convergence of the left hand side to $1$ for appropriately chosen $\eta_l,\eta_u>0$, for which it is enough to deduce $P_n^*\big(I_{\tau}(\hat{\mu}_{\tau,n}-\varepsilon,F_n^*)>\eta_l\big)\to 1$ and $P_n^*\big(I_{\tau}(\hat{\mu}_{\tau,n}+\varepsilon,F_n^*)<-\eta_u\big)\to 1$ almost surely. 
	Choose $2\,\eta_l= I_{\tau}(\mu_{\tau}-\varepsilon,F)\ne 0$ to obtain 
	\begin{align*}
		\prob_n^*\big(I_{\tau}(\hat \mu_{\tau,n}-\varepsilon,F_n^*)>\eta_l\big)&\geq \prob_n^*\big(\big| I_{\tau}(\mu_{\tau}-\varepsilon,F)-|I_{\tau}(\hat \mu_{\tau,n}-\varepsilon,F_n^*)-I_{\tau}(\mu_{\tau}-\varepsilon,F)|\big| >\eta_l \big)\\
		& \geq \prob_n^*\big(|I_{\tau}(\hat \mu_{\tau,n}-\varepsilon,F_n^*)-I_{\tau}(\mu_{\tau}-\varepsilon,F)|<\eta_l \big)
	\end{align*}
	by the inverse triangle inequality. Similar for $2\,\eta_u= I_{\tau}(\mu_{\tau}+\varepsilon,F)\ne 0$ we get that 
	\begin{align*}
		\prob_n^*\big(I_{\tau}(\mu^*_{\tau,n}+\varepsilon,F_n^*)<-\eta_u\big)&\geq \prob_n^*\big(|I_{\tau}(\mu^*_{\tau,n}+\varepsilon,F_n^*)-I_{\tau}(\mu_{\tau}+\varepsilon,F)|<\eta_u \big).
	\end{align*}
	In both inequalities the right hand side converges to $1$ almost surely, provided that almost surely, $I_{\tau}(\hat{\mu}_{\tau,n}\pm\varepsilon,F_n^*)\to I_{\tau}(\mu_{\tau}\pm\varepsilon,F)$ in probability conditionally on $Y_1,Y_2,\ldots$.
	To this end, start with \begin{equation*}
		\Ens{I_{\tau}(\hat{\mu}_{\tau,n}\pm\varepsilon,F_n^*)} = \sum_{i=1}^{n} \prob_n^*(Y_1^*=Y_i)\,I_{\tau}(\hat{\mu}_{\tau,n}\pm\varepsilon, Y_i) = I_{\tau}(\hat{\mu}_{\tau,n}\pm\varepsilon,F_n),
	\end{equation*}
so that it remains to show convergence of the right hand side to $I_{\tau}(\mu_{\tau}\pm\varepsilon,F)$ for almost every sequence $Y_1,Y_2,\ldots$. To this end it holds that \begin{equation*}
		\big|I_{\tau}(\hat{\mu}_{\tau,n}\pm\varepsilon,F_n)- I_{\tau}(\mu_{\tau}\pm\varepsilon,F_n) \big|\leq \big|\hat{\mu}_{\tau,n}-\mu_{\tau}\big|
	\end{equation*}
	by Lemma~\ref{lem:lipschitz}, where the bound converges to $0$ by the strong consistency of $\hat{\mu}_{\tau,n}$ \citep[Theorem~2]{holzmann2016}. Further $I_{\tau}(\mu_{\tau}\pm\varepsilon,F_n) \to I_{\tau}(\mu_{\tau}\pm\varepsilon,F)$ almost surely by the strong law of large numbers, thus $I_{\tau}(\hat{\mu}_{\tau,n}\pm\varepsilon,F_n)$ converges to $I_{\tau}(\mu_{\tau}\pm\varepsilon, F)$ for almost every sequence $Y_1,Y_2,\ldots$, what concludes the proof of individual consistency of $\mu^*_{\tau,n}$. 
		To strengthen this to uniform consistency, we use a Glivenko-Cantelli argument as in \citet{holzmann2016}, Theorem~2. 	
	Let $d_n=\hat{\mu}_{\tau_u,n}-\hat{\mu}_{\tau_l,n}$ and observe that $d_n\to d=\mu_{\tau_u}-\mu_{\tau_l}$ almost surely by consistency of $\hat{\mu}_{\tau,n}$. Let $m\in\N$ and choose $\tau_l=\tau_0\leq \tau_1\leq \ldots\leq \tau_m=\tau_u$ such that 
	\begin{equation*}
		\hat{\mu}_{\tau_k,n}=\hat{\mu}_{\tau_l,n}+\frac{k\,d_n}{m} \qquad (k\in\{1,\ldots,m\}),
	\end{equation*}
 which is possible because of the continuity of $\tau\mapsto\hat{\mu}_{\tau,n}$. Because the expectile functional is strictly increasing in $\tau$ it follows that 
\begin{align*}
		\mu_{\tau_{k},n}^*-\hat{\mu}_{\tau_{k+1},n}   \leq \mu^*_{\tau,n}-\hat{\mu}_{\tau,n} 
		\leq \mu_{\tau_{k+1},n}^*-\hat{\mu}_{\tau_{k},n} \qquad (\tau_k\leq \tau\leq \tau_{k+1}).
	\end{align*}
	This implies \begin{equation*}
		\sup_{\tau \in [\tau_l, \tau_u]} \big|\mu^*_{\tau,n} -\hat{\mu}_{\tau,n}\big| \leq \max_{1\leq k \leq m}\big|\mu_{\tau_{k},n}^*-\hat{\mu}_{\tau_{k},n} \big| + \frac{d_n}{m},
	\end{equation*}
	hence \begin{equation*}
		\limsup_{n} \| \mu_{\tau,n}^*-\hat{\mu}_{\tau,n}\|_{[\tau_l,\tau_u]}\leq \limsup_{n}\,\max_{1\leq k \leq m} \big|\mu_{\tau_{k},n}^*-\hat{\mu}_{\tau_{k},n} \big|+\limsup_{n}\frac{d_n}{m}=\frac{d}{m}
	\end{equation*}
	conditionally in probability for almost every sequence $Y_1,Y_2,\ldots$. Letting $m\to\infty$ completes the proof. 

\medskip
	
{\sl Proof of (\ref{eq:condconv}).}\quad The idea is to use Theorem~19.28, \cite{vdv2000}, for the random class \begin{align*}
		\mathcal{F}_n=\big\{y\mapsto -I_{\tau}(\hat{\mu}_{\tau,n},y)\,\mid\, \tau\in[\tau_l,\tau_u]\big\},
	\end{align*}
	which is a subset of \begin{align*}
		\mathcal{F}_{\eta_n}=\big\{y\mapsto -I_{\tau}(\mu_{\tau}+x,y)\,\mid\, |x|\leq \eta_n,\, \tau\in[\tau_l,\tau_u]\big\}
	\end{align*}
%
%
	for the sequence $\eta_n = \sup_{\tau \in [\tau_l, \tau_u]} |\hat{\mu}_{\tau,n}-\mu_{\tau}|$, hence, almost surely,  
	\begin{align*}
		J_{[\,]}(\varepsilon ,\mathcal{F}_{n},L_2(F_n)) \leq J_{[\,]}(\varepsilon,\mathcal{F}_{\eta_n},L_2(F_n)). 
	\end{align*}
The class $\mathcal{F}_{\eta_n}$ has envelope $(|\mu_{\tau_l}| + |\mu_{\tau_u}| + \eta_n + y)$, which satisfies the Lindeberg condition.

 By (\ref{eq:lip_x_plus_mu}) the class $\mathcal{F}_{\eta_n}$ is a class consisting of Lipschitz-functions with Lipschitz-constant $m_n(y)= (C+\eta_n+|y|)$ for some $C\geq 1$, such that the bracketing number fulfils \begin{equation*}
		N_{[\,]}(\delta, \mathcal{F}_{\eta_n}, L_2(F_n))\leq K\,\,\Big\{\En{m_n(Y)^2}\,\frac{\eta_n+\tau_u-\tau_l}{\delta}\Big\}^2
	\end{equation*}
	by Example~19.7, \citet{vdv2000}, where $K$ is some constant not depending on $n$. By the strong law of large numbers, $\En{m_n(Y)^2}\to \E{(C+|Y|)^2}$ holds almost surely, in addition $\eta_n\to 0$ almost surely by the strong consistency of $\hat{\mu}_{\tau,n}$, such that the above bracketing number is of order $1/\delta^2$. Thus the bracketing integral $J_{[\,]}(\varepsilon_n, \mathcal{F}_{\eta_n}, L_2(F_n))$ converges to $0$ almost surely for every sequence $\varepsilon_n\searrow 0$.
	
	Next we show the almost sure convergence of the expectation $\mathbb{E}_n\big[I_{\tau}(\hat{\mu}_{\tau,n},Y)I_{\tau'}(\hat{\mu}_{\tau',n},Y)\big]$ to $\E{I_{\tau}(\mu_{\tau},Y)I_{\tau'}(\mu_{\tau'},Y)}$. First it holds that \begin{align*}
		\En{I_{\tau}(\hat{\mu}_{\tau,n},Y)I_{\tau'}(\hat{\mu}_{\tau',n},Y) }&= \En{I_{\tau}(\hat{\mu}_{\tau,n},Y)\,\big(I_{\tau'}(\hat{\mu}_{\tau',n},Y)-I_{\tau'}(\mu_{\tau'},Y) \big)}\\
		&\quad+\En{I_{\tau'}(\mu_{\tau'},Y)\,\big(I_{\tau}(\hat{\mu}_{\tau,n},Y)-I_{\tau}(\mu_{\tau},Y) \big)}\\
		&\quad+\En{I_{\tau}(\mu_{\tau},Y)I_{\tau'}(\mu_{\tau'},Y)},
	\end{align*}
	where the last summand converges almost surely to $\E{I_{\tau}(\mu_{\tau},Y)I_{\tau'}(\mu_{\tau'},Y)}$ by the strong law of large numbers. For the first term we estimate \begin{align*}
		\En{|I_{\tau}(\hat{\mu}_{\tau,n},Y)|\,\big|I_{\tau'}(\hat{\mu}_{\tau',n},Y)-I_{\tau'}(\mu_{\tau'},Y) \big|} \leq \big\{|\hat{\mu}_{\tau,n}|+\En{|Y|} \big\}\,\big| \hat{\mu}_{\tau',n}-\mu_{\tau'}\big|
	\end{align*}
	with aid of Lemma~\ref{lem:lipschitz}, where the upper bound converges to $0$ almost surely by the strong law of large numbers, strong consistency of $\hat{\mu}_{\tau',n}$ and boundedness of $\hat{\mu}_{\tau,n}$, which in fact also follows from the strong consistency of the empirical expectile and since $\tau \in [\tau_l,\tau_u]$. The remaining summand above is treated likewise, hence $\En{I_{\tau}(\hat{\mu}_{\tau,n},Y)I_{\tau'}(\hat{\mu}_{\tau',n},Y)}$ indeed converges almost surely to the asserted limit. 
	The assertion now follows from Theorem~19.28, \citet{vdv2000}.

%

\medskip

{\sl Proof of (\ref{eq:developboot}).} \quad Set $\eta_n= \sup_{\tau \in [\tau_l, \tau_u]} |\hat{\mu}_{\tau,n}-\mu_{\tau}|$ again, then as a first step we can estimate \begin{align*}
		&\sup_{\| \varphi\| \leq \delta_n}\, \sup_{\tau \in [\tau_l, \tau_u]}\, \sqrt{n}\, \big|\psi_n^*(\hat{\mu}_{\cdot,n} + \varphi)(\tau) - \psi_n(\hat{\mu}_{\cdot,n} + \varphi)(\tau) - \big[ \psi_n^*(\hat{\mu}_{\cdot,n} )(\tau) - \psi_n(\hat{\mu}_{\cdot,n})(\tau) \big] \big|\\
		=& \sup_{\| \varphi\| \leq \delta_n}\, \sup_{\tau \in [\tau_l, \tau_u]}\, \sqrt{n}\, \big|\psi_n^*\big\{\mu_{\cdot}+(\hat{\mu}_{\cdot,n}-\mu_{\cdot} + \varphi)\big\}(\tau) - \psi_n\big\{\mu_{\cdot}+(\hat{\mu}_{\cdot,n}-\mu_{\cdot} + \varphi)\big\}(\tau) \\
		& \qquad \qquad \qquad \qquad \qquad\quad- \big[ \psi_n^*\big\{\mu_{\cdot}+(\hat{\mu}_{\cdot,n}-\mu_{\cdot})\big\}(\tau) - \psi_n\big\{\mu_{\cdot}+(\hat{\mu}_{\cdot,n}-\mu_{\cdot})\big\}(\tau) \big] \big|\\
		\leq & \sup_{| x_1|,|x_2| \leq \delta_n+\eta_n}\, \sup_{\tau \in [\tau_l, \tau_u]}\, \sqrt{n}\, \big|\psi_n^*(\mu_{\cdot} + x_1)(\tau) - \psi_n(\mu_{\cdot} + x_1)(\tau) \\
		& \qquad \qquad \qquad \qquad \qquad\quad- \big[ \psi_n^*(\mu_{\cdot}+x_2 )(\tau) - \psi_n(\mu_{\cdot}+x_2)(\tau) \big] \big|,
	\end{align*}
	so that for (\ref{eq:developboot}) it suffices to show the almost sure convergence $\Ens{\|\sqrt{n}(F_n^*-F_n)\|_{\mathcal{F}_{\rho_n}}}\to 0$ for the class \begin{align*}
		\mathcal{F}_{\rho_n} =\big\{y\mapsto I_{\tau}\big(\mu_{\tau}+x_1,y\big)-I_{\tau}\big(\mu_{\tau}+x_2,y\big) \,\big| \, |x_1|,|x_2|\leq \rho_n,\, \tau\in[\tau_l,\tau_u] \big\},~~ \rho_n=\delta_n+\eta_n.
	\end{align*}	
We use Corollary~19.35, \citet{vdv2000}, which implies that
 \begin{align*}
		\Ens{\|\sqrt{n}(F_n^*-F_n)\|_{\mathcal{F}_{\rho_n}}} \leq J_{[\,]}\big[\En{m_n(Y)^2},\mathcal{F}_{\rho_n},L_2(F_n)\big]
	\end{align*}
	almost surely, where $m_n(y)$ is an envelope function for $\mathcal{F}_{\rho_n}$. The proof consists of finding this envelope and determining the order of the bracketing integral. 	
	By using Lemma~\ref{lem:lipschitz} the class $\mathcal{F}_{\rho_n}$ consists of Lipschitz-functions, as for any $\tau, \tau'\in[\tau_l,\tau_u]$, $x_1,x_1',x_2,x_2'\in [-\rho_n,\rho_n]$ the almost sure inequality \begin{align}\label{eq:l_cont_boot}
	\big|&I_{\tau}(\mu_{\tau}+x_1,y)-I_{\tau}(\mu_{\tau}+x_2,y)-\big(I_{\tau'}(\mu_{\tau'}+x_1',y)-I_{\tau'}(\mu_{\tau'}+x_2',y)\big)\big|\notag\\
	&\leq \big(|x_1-x_1'|+|x_2-x_2'| + |\tau-\tau'|\big)\,2\,\big(C+\rho_n+|y|\big)
	\end{align}
	is true for some constant $C>0$, see also \eqref{eq:lip_x_plus_mu}. 
%
%
Thus  the bracketing integral $J_{[\,]}(\varepsilon_n, \mathcal{F}_{\scriptscriptstyle \rho_n}, L_2(F_n))$ converges to $0$ almost surely for any sequence $\varepsilon_n\to 0$ with the same arguments as above. 
Finally, by \eqref{eq:l_cont_boot} the function $m_n(y)=\eta_n\,4\,(C+|y|)$ is an envelope for $\mathcal{F}_{\rho_n}$. By the strong law of large numbers, the square integrability of $Y$ and since $\eta_n \to 0$ almost surely it holds that $\En{m_n(Y)^2}\to 0$ almost surely, so that  \begin{align*}
		\Ens{\|\sqrt{n}(F_n^*-F_n)\|_{\mathcal{F}_{\rho_n}}}\leq J_{[\,]}\big[\En{m_n(Y)^2},\mathcal{F}_{\rho_n},L_2(F_n)\big]\to 0
	\end{align*}
	is valid for almost every sequence $Y_1,Y_2,\ldots$, which concludes the proof.

\medskip

{\sl Proof of (\ref{eq:conv_psi_n_star}).}\quad By \eqref{eq:boot_consist} and \eqref{eq:developboot}, and since $\psi_n(\hat{\mu}_{\cdot,n}) = \psi_n^*(\mu_{\cdot,n}^*) = 0$, we have that almost surely, 
\begin{align*}
		\sqrt{n}\big(\psi_n(\mu_{\cdot,n}^*)-\psi_n(\hat{\mu}_{\cdot,n})\big) &= \sqrt{n}\big(\psi_n(\mu_{\cdot,n}^*)-\psi_n^*(\mu_{\cdot,n}^*)\big)\\
		&= - \sqrt{n}\big( \psi_n^*(\hat{\mu}_{\cdot,n})-\psi_n(\hat{\mu}_{\cdot,n})\big) + o_{\prob_n^*}(1)
	\end{align*}
The right hand side converges to $-Z$ conditionally in distribution, almost surely, by \eqref{eq:condconv}, which equals $Z$ in distribution.  
\end{proof}

\begin{lemma}\label{thm:asym_semi_hadam}

The map $\psi_n$ is invertible. Further,  	
 if $t_n\to 0$, $\varphi\in\mathcal{C}[\tau_l,\tau_u]$ and $\varphi_n\to\varphi$ with respect to $d_{hypi}$ and hence uniformly, we have that almost surely, conditionally on $Y_1, Y_2, \ldots $, 
\begin{equation}\label{eq:asympt_semi_hadam}
		t_n^{-1}\big(\psi_n^{\Inv}(t_n\varphi_n)-\psi_n^{\Inv}(0)\big) \to \dot{\psi}^{\Inv}(\varphi)
	\end{equation}
	with respect to the hypi-semimetric. 
\end{lemma}

%

\begin{proof}

The first part follows from Lemma \ref{lem:invertable} with $F$ in $\psi_0$ replaced by $F_n$ in $\psi_n$ as no specific assumptions on $F$ were used in that lemma.
 For (\ref{eq:asympt_semi_hadam}), with the same calculations as for Lemma~\ref{lem:rep_diff_psiinv} we obtain the representation \begin{align*}
		t_n^{-1}&\big(\psi_n^{\Inv}(t_n\varphi_n)-\psi_n^{\Inv}(0)\big)(\tau) \\
		&= \varphi_n(\tau)\Bigg\{\tau+(1-2\,\tau)\,\int_{0}^{1}F_n\big(\hat{\mu}_{\tau,n}+s\big(\psi_n^{\Inv}(t_n\,\varphi_n)(\tau)-\hat{\mu}_{\tau,n}\big)\big)\,\dif s\Bigg\}^{\Inv},
	\end{align*}
and we have to prove hypi convergence to $\dot{\psi}^{\Inv}(\varphi)$. 
By the same reductions as in the proof of Theorem~\ref{thm:semi_hadamard_diff}, it suffices to prove the hypi-convergence of \begin{equation*}
		h_n(\tau)=\int_{0}^{1}F_n\big(\hat{\mu}_{\tau,n}+s\big(\psi_n^{\Inv}(t_n\,\varphi_n)(\tau)-\hat{\mu}_{\tau,n}\big)\big)\,\dif s
	\end{equation*}
	to $F(\mu_{\tau})$ for almost every sequence $Y_1,Y_2,\ldots$. 	
	To this end, observe that for any $s\in[0,1]$ the sequence $\hat{\mu}_{\tau,n}+s\big(\psi_n^{\Inv}(t_n\,\varphi_n)(\tau)-\hat{\mu}_{\tau,n}\big)$ converges to $\mu_{\tau}$ almost surely by the same arguments as in Lemma~\ref{lem:rep_diff_psiinv}. Since $\mu_{\tau}$ is continuous in $\tau$, the almost sure convergence $\hat{\mu}_{\tau_n,n}+s\big(\psi_n^{\Inv}(t_n\,\varphi_n)(\tau_n)-\hat{\mu}_{\tau_n,n}\big)\to \mu_{\tau}$ holds for any sequence $\tau_n\to\tau$. By adding and subtracting $F_n(\hat{\mu}_{\tau,n}+s(\psi_n^{\Inv}(t_n\,\varphi_n)(\tau)-\hat{\mu}_{\tau,n}))$  and using Lemma~\ref{lem:rel_sup_inf}, we now can estimate \begin{align*}
		F(\mu_{\tau}-)&\leq \int_{0}^{1} \liminf_{n} \Bigg(F_n\big(\hat{\mu}_{\tau,n}+s\big(\psi_n^{\Inv}(t_n\,\varphi_n)(\tau)-\hat{\mu}_{\tau,n}\big)\big) \Bigg) \\
		+ \limsup_{n} & \Bigg(F\big(\hat{\mu}_{\tau,n}+s\big(\psi_n^{\Inv}(t_n\,\varphi_n)(\tau)-\hat{\mu}_{\tau,n}\big)\big) -F_n\big(\hat{\mu}_{\tau,n}+s\big(\psi_n^{\Inv}(t_n\,\varphi_n)(\tau)-\hat{\mu}_{\tau,n}\big)\big)  \Bigg)  \dif s\\
		&\leq \int_{0}^{1}\liminf_{n} F_n\big(\hat{\mu}_{\tau,n}+s\big(\psi_n^{\Inv}(t_n\,\varphi_n)(\tau)-\hat{\mu}_{\tau,n}\big)\big) \,\dif s + \limsup_{n} \| F_n-F\|_{\R}\\
		&\leq \liminf_{n} \int_{0}^{1} F_n\big(\hat{\mu}_{\tau,n}+s\big(\psi_n^{\Inv}(t_n\,\varphi_n)(\tau)-\hat{\mu}_{\tau,n}\big)\big) \,\dif s = \liminf_{n} h_n(\tau_n)
	\end{align*}
	almost surely, where the '$\limsup$' vanishes due to the Glivenko-Cantelli-Theorem for the empirical distribution function. Similarly we almost surely have \begin{align*}
		F(\mu_{\tau}) &\geq \int_{0}^{1}\limsup_{n} F_n\big(\hat{\mu}_{\tau,n}+s\big(\psi_n^{\Inv}(t_n\,\varphi_n)(\tau)-\hat{\mu}_{\tau,n}\big)\big) \,\dif s + \liminf_{n} \| F_n-F\|_{\R}\\
		&\geq \limsup_{n} \int_{0}^{1} F_n\big(\hat{\mu}_{\tau,n}+s\big(\psi_n^{\Inv}(t_n\,\varphi_n)(\tau)-\hat{\mu}_{\tau,n}\big)\big) \,\dif s = \limsup_{n} h_n(\tau_n).
	\end{align*}
The proof is concluded as that of Theorem~\ref{thm:semi_hadamard_diff} by using Corollary~A.7, \citet{buecher2014} .
\end{proof}

\begin{proof}[Proof of Theorem~\ref{thm:main_thm_boot}.]
	
	Set $t_n=1/\sqrt{n}$ and define the function $g_n(\varphi)= t_n^{-1}\,(\psi_n^{\Inv}(t_n\,\varphi)-\psi_n^{\Inv}(0))$. Then from (\ref{eq:asympt_semi_hadam}) the hypi-convergence $g_n(\varphi_n)\to \dot{\psi}^{\Inv}(\varphi)$ holds almost surely, whenever $\varphi\in\mathcal{C}[\tau_l,\tau_u]$ and $\varphi_n\to\varphi$ with respect to $d_{hypi}$. In addition $\sqrt{n}\big(\psi_n(\mu_{\cdot, n}^*)-\psi_n(\hat{\mu}_{\cdot,n})\big)\to Z$ conditional in distribution with respect to the sup-norm, almost surely, by \eqref{eq:condconv}, where $Z$ is continuous almost surely. Hence the convergence is also valid with respect to $d_{hypi}$, such that \begin{equation*}
		\sqrt{n}\big( \mu_{\cdot, n}^*-\hat{\mu}_{\cdot,n}\big) = g_n\big(\sqrt{n}\big(\psi_n(\mu_{\cdot, n}^*)-\psi_n(\hat{\mu}_{\cdot,n})\big)\big) \to \dot{\psi}^{\Inv}(Z)
	\end{equation*}
	holds conditionally in distribution, almost surely, by using the extended continuous mapping theorem, Theorem~B.3, in \citet{buecher2014}. 
\end{proof}

\subsection{Details for the proof of Theorem \ref{thm:conv_quantile}}

Since $F$ in Theorem \ref{thm:conv_quantile} is assumed to be continuous and strictly increasing, it is differentiable almost everywhere. If $F$ is not differentiable in a point $x$, the assumptions at least guarantee right- and left-sided limits of $f$ in $x$. Since redefining a density on a set of measure zero is possible without changing the density property, we can assume $f(x)=f(x-)$ for such $x$. 
In addition, the assumptions imply that $F$ is invertible with continuous inverse, which we denote by $F^{\Inv}$.

\begin{proof}[Proof of Lemma~\ref{lem:uniform_semi_hadam}.]
	Let $t_n\searrow 0$, $\alpha_n\to\alpha\in[\alpha_l,\alpha_u]$ and $\varphi_n, \nu_n\in\ell^{\infty}[\alpha_l,\alpha_u]$ with $\varphi_n\to\varphi\in\mathcal{C}[\alpha_l,\alpha_u]$ and $\nu_n\to 0$ with respect to $\dif_{hypi}$. Then $\varphi_n(\alpha_n)\to\varphi(\alpha)$ and $\nu_n(\alpha_n)\to 0$ by Proposition~2.1, \citet{buecher2014}. 
	We have to deal with the limes inferior and superior of \begin{equation*}
		t_n^{-1}\big(F^{\Inv}(\alpha_n+\nu_n(\alpha_n)+t_n\,\varphi_n(\alpha_n))-F^{\Inv}(\alpha_n+\nu_n(\alpha_{n}))\big).
	\end{equation*}
	Define the function $g:[0,1]\to\R$ with $g(s) = F^{\Inv}(\alpha_n+\nu_n(\alpha_n)+s\,t_n\,\varphi_n(\alpha_n))$. The difference above then equals $t_n^{-1}(g(1)-g(0))$. 
	The function $g$ is monotonically increasing (or decreasing, depending on the sign of $\varphi_n(\alpha_n)$) and continuous in $s$. Hence with Theorem~7.23, \citet{bruckner2008}, we can write \begin{align}
		t_n^{-1}\big(g(1)-g(0)\big) &=\int_{0}^{1} g'(s)\,\dif s +\mu_g\big(\{g'=\pm\infty\}\big)\notag\\
		&= \varphi_n(\alpha_n)\,\int_{0}^{1} \Big[f\big\{F^{\Inv}\big(\alpha_n+\nu_n(\alpha_n)+s\,t_n\,\varphi_n(\alpha_n)\big)\big\}\Big]^{-1}\,\dif s,\label{eq:diff_quot}
	\end{align}
	where $\mu_g$ is the Lebesgue-Stieltjes signed measure associated with $g$. The set $\{g'=\pm\infty\}$ is empty, as $g'=\infty$ if and only if $f=0$, which is excluded by assumption, so that $\lambda_g\big(\{g'=\pm\infty\}\big)=0$. Since $\varphi_n\to \varphi \in\mathcal{C}[\alpha_l,\alpha_u]$, using Lemma~\ref{lem:diverses} i) we only need to deal with the accumulation points of the integral in \eqref{eq:diff_quot}. To this end let
	\begin{equation*}
		H_n(\alpha) = \int_{0}^{1} \Big(f\big\{F^{\Inv}\big(\alpha_n+\nu_n(\alpha_n)+s\,t_n\,\varphi_n(\alpha_n)\big)\big\}\Big)^{-1}\,\dif s.
	\end{equation*}
	The assertion is that $H_n(\alpha)$ hypi-converges to $f(q_{\alpha})^{-1}$. 
	In the case for the expectile-process we were able to exploit continuity properties of the asserted semi-derivative  and thus utilize Corollary~A.7, \citet{buecher2014}). In the present context we do not know whether the limit $f(q_{\cdot})^{-1}$ can be attained by extending $f(q_{\cdot})^{-1}$ from the dense set where it is continuous. Thus, we directly show the defining properties of hypi convergence in (\ref{eq:epi_pointwise}) and (\ref{eq:hypo_pointwise}).  
	For any $s\in [0,1]$ it holds that $F^{\Inv}\big(\alpha_n+\nu_n(\alpha_n)+s\,t_n\,\varphi_n(\alpha_n)\big)\to q_{\alpha}$ by continuity of $F^{\Inv}$, hence we can estimate \begin{align*}
		\liminf_{n} &\int_{0}^{1}\Big(f\big\{F^{\Inv}\big(\alpha_n+\nu_n(\alpha_n)+s\,t_n\,\varphi_n(\alpha_n)\big)\big\}\Big)^{-1}\,\dif s\\
		\geq &\int_{0}^{1} \Big(\limsup_{n} f\big\{F^{\Inv}\big(\alpha_n+\nu_n(\alpha_n)+s\,t_n\,\varphi_n(\alpha_n)\big)\big\}\Big)^{-1}\,\dif s
		\geq \frac{1}{f_{\vee}(q_{\alpha})}.
	\end{align*}
	with help of Fatou's Lemma, Lemma~\ref{lem:rel_sup_inf} and the properties of the hulls. Similarly, 
	\begin{align*}
		\limsup_{n} &\int_{0}^{1}\Big(f\big\{F^{\Inv}\big(\alpha_n+\nu_n(\alpha_n)+s\,t_n\,\varphi_n(\alpha_n)\big)\big\}\Big)^{-1}\,\dif s\\
		\leq &\int_{0}^{1} \Big(\liminf_{n} f\big\{F^{\Inv}\big(\alpha_n+\nu_n(\alpha_n)+s\,t_n\,\varphi_n(\alpha_n)\big)\big\}\Big)^{-1}\,\dif s \leq \frac{1}{f_{\wedge}(q_{\alpha})}.
	\end{align*}
It remains to construct sequences $\alpha_n$ such that $\lim_{n} H_n(\alpha_n)$ equals the respective hull. 	
	By Lemma~\ref{lem:diverses} ii), it holds that 
	\begin{align*}
		f_{\wedge}(q_{\alpha}) = \min\big\{f(q_{\alpha}-),f(q_{\alpha}+),f(q_{\alpha}) \big\}\quad\text{and}\quad f(q_{\alpha}) = \max\big\{f(q_{\alpha}-),f(q_{\alpha}+),f(q_{\alpha}) \big\}.
	\end{align*}
	If $F$ is differentiable in $q_{\alpha}$, all three values can differ; if not, we have set $f(q_{\alpha})=f(q_{\alpha}-)$. 	
We shall choose a sequence converging from below, above or which is always equal to $\alpha$, depending on where the $\min/\max$ is attained. 	
	We do this exemplary for $f_{\vee}(q_{\alpha}) =f(q_{\alpha}-)$. Since $\varphi_n\to\varphi$ uniformly and $\varphi$ is bounded, there is a $C>0$ for which $\Vert\varphi_n\Vert\leq C$. Then choose any sequence $\varepsilon_n\searrow0$ and set $\alpha_n = \alpha-\|\nu_n\|-t_n\,C-\varepsilon_n$, which is in $[\alpha_l,\alpha_u]$ for $n$ big enough. For this sequence it holds that $\alpha_n +\nu_n(\alpha_n)+ s\,t_n\,\varphi_n(\alpha_n) < \alpha$ for any $s\in[0,1]$, so that the convergence is from below. Bounded convergence yields 
	\begin{align*}
		& \lim_n \int_{0}^{1}\Big(f\big\{F^{\Inv}\big(\alpha_n+\nu_n(\alpha_n)+s\,t_n\,\varphi_n(\alpha_n)\big)\big\}\Big)^{-1}\,\dif s\\
		= & \int_{0}^{1}\Big(\lim_nf\big\{F^{\Inv}\big(\alpha_n+\nu_n(\alpha_n)+s\,t_n\,\varphi_n(\alpha_n)\big)\big\}\Big)^{-1}\,\dif s = \frac{1}{f(q_{\alpha}-)}.
	\end{align*}
\end{proof}

\subsection{Properties of hypi-convergence}

An major technical issue in the above argument is to determine hypi-convergence of sums, products and quotients of hypi-convergent functions. We shall require the following basic relations between $\limsup$ and $\liminf$. 

\begin{lemma}\label{lem:rel_sup_inf}
	
	Let $(a_n)_n$, $(b_n)_n $ be bounded sequences. Then 
	\begin{align*}
		\liminf_{n\to\infty} a_n + \limsup_{n\to\infty}b_n \geq \liminf_{n\to\infty} (a_n+b_n) \geq \liminf_{n\to\infty} a_n+\liminf_{n\to\infty}b_n
	\end{align*}
	and \begin{align*}
		\limsup_{n\to\infty} a_n + \liminf_{n\to\infty}b_n\leq \limsup_{n\to\infty} (a_n+b_n) \leq \limsup_{n\to\infty} a_n+\limsup_{n\to\infty}b_n. 
	\end{align*}
	If $a_n \ne 0$, then 
	\begin{align*}
		\liminf_{n} \frac{1}{a_n} = \frac{1}{\limsup_{n} a_n}.
	\end{align*}
	
%
\end{lemma}

\begin{proof}
	
	The first pair of inequalities follows from
	\begin{align*}
		a_k + \sup_{l\geq n} b_l \geq a_k + b_k \geq a_k + \inf_{l\geq n} b_l
	\end{align*} 
	for any $n\in\N$, $k\geq n$, by applying '$\inf_{k\geq n}$' and taking the limit $n\to\infty$. 
		The second pair of inequalities follows similarly.  
	For the last part note that \begin{equation*}
		\inf_{l\geq n} \frac{1}{a_l} = \frac{1}{\sup_{l\geq n} a_l}. 
	\end{equation*}
Taking $n\to\infty$ yields the asserted equality.
	%
%
\end{proof}

\begin{lemma}\label{lem:acc_points}
	Let $(b_n)_n$ be a bounded sequence and let $(a_n)_n$ be convergent with limit $a \in \R$. Then 
	\begin{align*}
		\liminf_n a_n\,b_n &= \liminf_n a\,b_n, \qquad 
		\limsup_n a_n\,b_n = \limsup_n a\,b_n.
	\end{align*}
\end{lemma}

\begin{proof}
From Lemma \ref{lem:rel_sup_inf}, 
\small
\begin{align*}
\liminf_n (a_0\,b_n-a_n\,b_n) + \liminf_n a_0\, b_n & \leq \liminf_n a_n\,b_n  \leq  \limsup_n (a_0\,b_n-a_n\,b_n) + \liminf_n a_0\, b_n,\\
\liminf_n (a_0\,b_n-a_n\,b_n) + \limsup_n a_0\,b_n  & \leq \limsup_n a_n\,b_n
 \leq 		\limsup_n (a_0\,b_n-a_n\,b_n) + \limsup_n a_0\,b_n.
	\end{align*}
\normalsize
But
$		\big|a_0\,b_n-a_n\,b_n\big| \leq \sup_m |b_m|\,\big|a_0-a_n\big|\to 0,
$ which implies the stated result. 
\end{proof}

The next lemma was crucial in the proof of Theorem~\ref{thm:semi_hadamard_diff}.

\begin{lemma}\label{lem:diverses}
	Let $c, c_n, \varphi_n\in\ell^{\infty}\big([l,u]\big)$ and $\varphi\in\mathcal{C}[l,u]$.
	
	\begin{enumerate}[label={\roman*)}]
		\item If $d_{hypi}(\varphi_n,\varphi),\, d_{hypi}(c_n,c)\to 0$ holds, then $c_n\,\varphi_n$ hypi-converges to $c\,\varphi$. More precisely $c_n\,\varphi_n$ epi-converges to $(c\,\varphi)_{\wedge}$ and hypo-converges to $(c\,\varphi)_{\vee}$, where
		\begin{align}\label{eq:hull_product}
			\big(\varphi\,c\big)_{\wedge} &= \varphi\,\big(c_{\wedge}\,1_{\varphi>0} + c_{\vee}\,1_{\varphi<0}\big), \qquad 
			\big(\varphi\,c\big)_{\vee} = \varphi\,\big(c_{\vee}\,1_{\varphi>0} + c_{\wedge}\,1_{\varphi<0}\big).
	\end{align}
		\item Assume $c$ has left- and right-sided limits at every point in $[l,u]$. Then \begin{align*}
			c_{\wedge}(x)&=\min\{c(x-),c(x),c(x+)\} \qquad\text{and}\\
			c_{\vee}(x)&=\max\{c(x-),c(x),c(x+)\}
		\end{align*}
		holds. Especially, if $c(x_0-)\leq c(x_0)\leq c(x_0+)$ or $c(x_0-)\geq c(x_0)\geq c(x_0+)$ for some $x_0\in [l,u]$, then the equalities $(c_{\wedge})_{\vee}(x_0)=c_{\vee}(x_0)$ and $(c_{\vee})_{\wedge}(x_0)=c_{\wedge}(x_0)$ are true.
		
		\item If $c_n,\,c>0$, the convergence $d_{hypi}(\nicefrac{1}{c_n},\nicefrac{1}{c})\to 0$ follows from $d_{hypi}(c_n,c)\to 0$. 
	\end{enumerate}
\end{lemma}

\begin{proof}[Proof of Lemma~\ref{lem:diverses}.]
From the definition in (\ref{eq:lowerupperhulls}), for a function $h\in\ell^{\infty}[l,u]$ the lower semi-continuous hull $h_\wedge$ at $x \in [l,u]$ is characterized by the following conditions
\begin{align}\label{eq:lowerhull}
\begin{split}
\text{For any sequence } x_n \to x, \quad \text{we have } & \liminf_{n \to \infty} h(x_n) \geq h_\wedge(x),\\
 \text{There is a sequence } x_n' \to x \quad \text{for which } & \lim_{n \to \infty} h(x_n') = h_\wedge(x), 
\end{split}
\end{align}
and similarly for $h_\vee$.

\medskip

Ad i):\quad By continuity of $\varphi$ we have $ \varphi(x_n) \to \varphi(x)$ for any sequence $x_n \to x$. The statement (\ref{eq:hull_product}) now follows immediately using (\ref{eq:lowerhull}) and Lemma \ref{lem:acc_points} and noting that for $\varphi(x) < 0$, 
\[ \liminf_{n \to \infty} \varphi(x) c(x_n) = \varphi(x)\, \limsup_{n \to \infty}  c(x_n),\quad \limsup_{n \to \infty} \varphi(x) c(x_n) = \varphi(x)\, \liminf_{n \to \infty}  c(x_n).  \] 
Further, by continuity of $\varphi$, the hypi-convergence of $\varphi_n$ to $\varphi$ actually implies the uniform convergence. Therefore, for any $x_n \to x$ we have that $ \varphi_n(x_n) \to \varphi(x)$. Using the pointwise criteria (\ref{eq:epi_pointwise}) and (\ref{eq:hypo_pointwise}) for hypi-convergence, Lemma \ref{lem:acc_points} and (\ref{eq:hull_product}) we obtain the asserted convergence $\varphi_n\, c_n \to \varphi\, c$ with respect to the hypi semi-metric. 

\medskip
	
Ad ii): The proof of Lemma~C.6, \citet{buecher2014} show that for a function, which admits right- and left-sided limits, the supremum over a shrinking neighbourhood around a point $x$ converges to the maximum of the three points $c(x-),\,c(x)$ and $c(x+)$. The analogues statement holds for the infimum, which is the first part of ii). 
From Lemma~C.5, \citet{buecher2014}, the maps $x\mapsto c(x-)$ and $x\mapsto c(x+)$ both have a right-sided limit equal to $c(x+)$ and a left-sided limit equal to $c(x-)$, hence this is also true for the functions $c_{\vee}(x)=\max\{c(x-),c(x), c(x+)\}$ and $c_{\wedge}(x)= \min\{c(x-),c(x) c(x+)\}$ . From the above argument, we obtain
\begin{align*}
		(c_{\vee})_{\wedge}(x)&=\min\big\{c(x-),c(x+), \max\{c(x-),c(x), c(x+) \} \big\} = \min\{c(x-),c(x+)\} \qquad \text{and}\\
		(c_{\wedge})_{\vee}(x)&=\max\big\{c(x-),c(x+), \min\{c(x-),c(x), c(x+) \} \big\} = \max\{c(x-),c(x+)\}.
	\end{align*}
	If $c(x_0-)\leq c(x_0)\leq c(x_0+)$ or $c(x_0-)\geq c(x_0)\geq c(x_0+)$, we obtain
	\begin{align*}
		c_{\vee}(x_0)&=\max\{c(x-), c(x+)\} = (c_{\wedge})_{\vee}(x_0),\quad		c_{\wedge}(x_0) = \min\{c(x-), c(x+)\}=(c_{\vee})_{\wedge}(x_0).
	\end{align*}

\medskip

Ad iii): 	From Lemma~\ref{lem:rel_sup_inf} (last statement) and (\ref{eq:lowerhull}) we obtain $(\nicefrac{1}{c})_{\wedge}=\nicefrac{1}{c_{\vee}}$ and $(\nicefrac{1}{c})_{\vee}=\nicefrac{1}{c_{\wedge}}$. The hypi-convergence of 
$\nicefrac{1}{c_n}$ to these hulls follows similarly from Lemma~\ref{lem:rel_sup_inf} (last statement) and the pointwise criteria (\ref{eq:epi_pointwise}) and (\ref{eq:hypo_pointwise}) for hypi-convergence.  
\end{proof}

\end{document}